\documentclass[a4paper,11pt]{article}

\usepackage[margin=0.8in]{geometry}

\usepackage{braket}

\usepackage[backref=page]{hyperref}

\usepackage{amssymb,mathtools,amsthm,amsmath,bm}
\usepackage{tikz}
\usetikzlibrary{calc,patterns,angles,quotes}
\usepackage{subcaption}  % For subfigures
\usepackage{xargs}
\usepackage{gensymb}

\usepackage{enumitem}
\usepackage{verbatim}
\usepackage{thm-restate}

\theoremstyle{plain}
\newtheorem{theorem}{\bf Theorem}[section]
\newtheorem{claim}[theorem]{\bf Claim}
\newtheorem{conjecture}[theorem]{\bf Conjecture}
\newtheorem{proposition}[theorem]{\bf Proposition}

\newtheorem{corollary}[theorem]{\bf Corollary}
\newtheorem{lemma}[theorem]{\bf Lemma}

\theoremstyle{definition}

\newenvironment{remark}[1][Remark.]{\begin{trivlist}
		\item[\hskip \labelsep {\bfseries #1}]}{\end{trivlist}}

\newcommand{\Rea}{{\mathbb R}}

\DeclareMathOperator{\rank}{\text{rank}}

\DeclareMathOperator{\T}{\top}

\newcommand{\lfrac}[2]{\left\lfloor\frac{#1}{#2}\right\rfloor}

\newcommand{\blowup}[2]{ {#1}^{(#2)} }

\newcommand{\lminusva}[2]{L^{#1 \downarrow}_{#2}}

\DeclareMathOperator{\Ima}{Im}

\DeclareMathOperator{\mult}{\text{mult}}
\DeclareMathOperator{\Spec}{\text{Spec}}

\usepackage[nobysame,msc-links,non-sorted-cites]{amsrefs}

\renewcommand{\eprint}[1]{\href{https://arxiv.org/abs/#1}{arXiv:#1}}

\BibSpec{book}{%
	+{}  {\PrintPrimary}                {transition}
	+{,} { \textbf}                     {title} % was \textit
	+{.} { }                            {part}
	+{:} { \textit}                     {subtitle}
	+{,} { \PrintEdition}               {edition}
	+{}  { \PrintEditorsB}              {editor}
	+{,} { \PrintTranslatorsC}          {translator}
	+{,} { \PrintContributions}         {contribution}
	+{,} { }                            {series}
	+{,} { \voltext}                    {volume}
	+{,} { }                            {publisher}
	+{,} { }                            {organization}
	+{,} { }                            {address}
	+{,} { \PrintDateB}                 {date}
	+{,} { }                            {status}
	+{}  { \parenthesize}               {language}
	+{}  { \PrintTranslation}           {translation}
	+{;} { \PrintReprint}               {reprint}
	+{.} { }                            {note}
	+{.} {}                             {transition}
	+{}  {\SentenceSpace \PrintReviews} {review}
}
\BibSpec{incollection}{%
  +{}  {\PrintAuthors}                {author}
  +{,} { \textit}                     {title}
  +{.} { }                            {part}
  +{:} { \textit}                     {subtitle}
  +{,} { \PrintContributions}         {contribution}
  +{,} { \PrintConference}            {conference}
  +{}  {\PrintBook}                   {book}
  +{,} { }                            {booktitle}
	+{}  { \PrintEditorsB}              {editor}
	+{,} { }                            {publisher}
  +{,} { \PrintDateB}                 {date}
  +{,} { pp.~}                        {pages}
  +{,} { }                            {status}
  +{,} { \PrintDOI}                   {doi}
  +{,} { available at \eprint}        {eprint}
  +{}  { \parenthesize}               {language}
  +{}  { \PrintTranslation}           {translation}
  +{;} { \PrintReprint}               {reprint}
  +{.} { }                            {note}
  +{.} {}                             {transition}
  +{}  {\SentenceSpace \PrintReviews} {review}
}

\makeatletter
\renewcommand{\PrintNames@a}[4]{%
	\PrintSeries{\name}
	{#1}
	{}{ and \set@othername}
	{,}{ \set@othername}
	{}{ and \set@othername}
	{#2}{#4}{#3}%
}
\makeatother

  \makeatletter
\def\@fnsymbol#1{\ensuremath{\ifcase#1\or *\or **\or 
   \mathsection\or \mathparagraph\or \|\or  \dagger\dagger
   \or \ddagger\ddagger \else\@ctrerr\fi}}
    \makeatother

\usepackage{authblk}

\begin{document}

\title{Stiffness matrices of graph blow-ups and the $d$-dimensional algebraic connectivity of complete bipartite graphs}

\author{Yunseong Jung\thanks{Department of Mathematics, University of California, San Diego, La Jolla CA 92093, USA. Email: \href{mailto:y8jung@ucsd.edu}{y8jung@ucsd.edu}. Research supported by the Porges Family Fund for Undergraduate Research through the Office of Undergraduate Research and Scholar Development at Carnegie Mellon University.
}
}

    \author{Alan Lew\thanks{Dept. Math. Sciences, Carnegie Mellon University, Pittsburgh, PA 15213, USA. Email: \href{mailto:alanlew@andrew.cmu.edu}{alanlew@andrew.cmu.edu}.}}

    \affil{}
   
    \date{}
\maketitle

\begin{abstract}

The $d$-dimensional algebraic connectivity $a_d(G)$ of a graph $G=(V,E)$ is a quantitative measure of its $d$-dimensional rigidity, defined in terms of the eigenvalues of stiffness matrices associated with different embeddings of the graph into $\mathbb{R}^d$. 
For a function $a:V\to \mathbb{N}$, we denote by $G^{(a)}$ the $a$-blow-up of $G$, that is, the graph obtained from $G$ by replacing every vertex $v\in V$ with an independent set of size $a(v)$. We determine a relation between the stiffness matrix eigenvalues of $G^{(a)}$ and the eigenvalues of certain weighted stiffness matrices associated with the original graph $G$. This resolves, as a special case, a conjecture of Lew, Nevo, Peled and Raz on the stiffness eigenvalues of balanced blow-ups of the complete graph.

As an application, we obtain a lower bound on the $d$-dimensional algebraic connectivity of complete bipartite graphs. More precisely, we prove the following: Let $K_{n,m}$ be the complete bipartite graph with sides of size $n$ and $m$ respectively. Then, for every $d\ge 1$ there exists $c_d>0$ such that, for all $n,m\ge d+1$ with $n+m\ge \binom{d+2}{2}$, $a_d(K_{n,m})\geq c_d\cdot \min\{n,m\}$. This bound is tight up to the multiplicative constant. In the special case $d=2$, $n=m=3$, we obtain the improved bound $a_2(K_{3,3})\ge 2(1-\lambda)$, where $\lambda\approx 0.6903845$ is the unique positive real root of the polynomial $176 x^4-200 x^3+47 x^2+18 x-9$, which we conjecture to be tight.
\end{abstract}

\section{Introduction}

Let $d\ge 1$. A \textit{$d$-dimensional framework} is a pair $(G,p)$, where $G=(V,E)$ is a finite, simple graph, and $p$ is a map from $V$ to $\mathbb{R}^d$. The framework $(G,p)$ is called \textit{rigid} if there is no continuous motion of the vertices, starting from the positions specified by $p$, that preserves the distance between all pairs of adjacent vertices (except for the trivial motions --- rotations and translations of the whole graph).

An embedding $p:V\to\Rea^d$ is called \emph{generic} if the $d|V|$ coordinates of $p$ are algebraically independent over the rationals. It was shown by Asimow and Roth in \cite{AR78} that, for every graph $G$, if $(G,p)$ is rigid for some generic embedding $p:V\to\Rea^d$, then $(G,p)$ is rigid for \emph{all} $d$-dimensional generic embeddings. In such a case, we say that $G$ is \emph{$d$-rigid}.

Let $(G,p)$ be a $d$-dimensional framework. For a pair vertices $u,v\in V$, we define $d_{uv}\in \Rea^d$ by
\[
    d_{uv}=\begin{cases}
        \frac{p(u)-p(v)}{\|p(u)-p(v)\|} & \text{if } p(u)\ne p(v),\\
        0 &\text{otherwise.}
    \end{cases}
\]
The (normalized) \emph{rigidity matrix} of $(G,p)$ is the matrix $R(G,p)\in \Rea^{d|V|\times |E|}$ defined as follows. Let the rows of $R(G,p)$ be indexed by pairs $(u,i)$, where $u\in V$ and $i\in[d]=\{1,2,\ldots,d\}$, and let its columns be indexed by the edges $e\in E$. Then, we define
\[
R(G,p)_{(u,i),e}= \begin{cases}
    (d_{uv})_i & \text{if } e=\{u,v\} \text{ for some } v\in V,\\
    0 & \text{otherwise,}
\end{cases}
\]
for all $u\in V$, $i\in[d]$ and $e\in E$.
For simplicity, we will assume from now on that the image of $V$ under $p$ is $d$-dimensional (that is, it is not contained in any hyperplane), and in particular $|V|\ge d+1$. 
It is well known (see \cite{AR78}) that the rank of $R(G,p)$ is always at most $d|V|-\binom{d+1}{2}$. The framework $(G,p)$ is called \textit{infinitesimally rigid} if the rank of $R(G,p)$ is exactly $d|V| - \binom{d+1}{2}$. 
Every infinitesimally rigid framework is rigid (see \cite{Gluck1974}). Moreover, for a generic embedding $p$, $(G,p)$ is rigid if and only if it is infinitesimally rigid. In fact, a graph $G=(V,E)$ with at least $d+1$ vertices is $d$-rigid if and only if $\rank(R(G,p))=d|V|-\binom{d+1}{2}$ for some map $p:V\to\Rea^d$ (see \cite{AR78}).

The \emph{stiffness matrix} $L(G,p)$ is defined as 
\[
L(G,p) = R(G,p)R(G,p)^{\T} \in \Rea^{d|V| \times d|V|}.
\]
Since $L(G,p)$ is positive semi-definite, all its eigenvalues are non-negative real numbers. For $1\le k\le d|V|$, we denote by $\lambda_k(L(G,p))$ the $k$-th smallest eigenvalue of $L(G,p)$.   Since $\rank(L(G,p)) = \rank(R(G,p)) \leq d|V| - \binom{d+1}{2}$, the kernel of $L(G,p)$ has dimension at least $\binom{d+1}{2}$. Hence,
\[
\lambda_1(L(G,p)) = \lambda_2(L(G,p)) = \cdots = \lambda_{\binom{d+1}{2}}(L(G,p)) = 0.
\]
We call the next eigenvalue, $\lambda_{{d+1\choose 2}+1}(L(G,p))$, the \textit{spectral gap} of $(G,p)$. Jord\'{a}n and Tanigawa defined in \cite{JT22} the \textit{$d$-dimensional algebraic connectivity} of $G$ as
\[
a_d(G) = \sup\left\{\lambda_{{d+1\choose 2}+1}(L(G,p))\,  \bigg\vert\,  p: V \rightarrow \Rea^d\right\}.
\]
Note that $a_d(G)>0$ if and only if $G$ is $d$-rigid. We may think of the $d$-dimensional algebraic connectivity as a quantitative measure of the rigidity of the graph. For example, in \cite{JT22}, Jord\'{a}n and Tanigawa showed that dense enough random subgraphs of graphs with large $d$-dimensional algebraic connectivity are, with high probability, $d$-rigid (see \cite[Corollary 8.2]{JT22} for a precise statement). 

For $d=1$ and any embedding $p: V \rightarrow \Rea$, $L(G,p)$ is equal to the Laplacian matrix $L(G)$ of the graph $G$. Therefore, $a_1(G)=\lambda_2(L(G))$ coincides with the classical notion of the algebraic connectivity of $G$, introduced by Fiedler in \cite{Fie73}.

It is generally difficult to determine the exact value of the $d$-dimensional algebraic connectivity of a graph. For example, let $K_n$ be the complete graph on $n$ vertices.
It is well known and easy to prove that $a_1(K_n)=n$. Jord{\'a}n and Tanigawa showed in \cite{JT22}, building on previous results by Zhu \cite{zhu2013quantitative}, that $a_2(K_n)=n/2$. However, for $d\ge 3$, only partial results are known: In \cite{LNPR2023onthe} it was shown that $a_d(K_{d+1})=1$ for all $d\ge 3$. For general $n\ge d+1$, the best currently known bounds, proved in \cite{LNPR2023onthe,LNPR23+}, are 
\[
\frac{1}{2}\cdot \lfrac{n}{d} \leq a_d(K_n) \leq \frac{2n}{3(d-1)} + \frac{1}{3}.
\]

\subsection{Weighted rigidity matrices and the blow-up of a framework}

Let $G=(V,E)$ be a graph, and let  $p:V\to\Rea^d$ and $f:V\to \Rea_{>0}$. We define the \emph{weighted rigidity matrix} $R_f(G,p)\in \Rea^{d|V|\times |E|}$ by
\[
R_f(G,p)_{(u,i),e}=\begin{cases}
    \sqrt{f(v)}(d_{uv})_i & \text{if } e=\{u,v\} \text{ for some } v\in V,\\
    0 & \text{otherwise},
\end{cases}
\]
for all $u\in V$, $i\in[d]$ and $e\in E$.  
Let $L_f(G,p)=R_f(G,p) R_f(G,p)^{\T}\in \Rea^{d|V|\times d|V|}$.

For a vertex $v\in V$, let $N_G(v)$ denote the set of neighbors of $v$ in $G$, and let $\deg_G(v)=|N_G(v)|$. 
Let $R^{v}_f(G,p)\in \Rea^{d\times \deg_G(v)}$ be defined by
\[
    R^{v}_f(G,p)_{i,u}=\sqrt{f(u)}(d_{vu})_i
\]
for all $i\in[d]$ and $u\in N_G(v)$.
That is, for every $u\in N_G(v)$, the column of $R^v_f(G,p)$ indexed by $u$ is the vector $\sqrt{f(u)}d_{vu}$. Let $L^v_f(G,p)=R^{v}_f(G,p) R^{v}_f(G,p)^{\T}\in \Rea^{d\times d}$. 
If $f\equiv 1$, we denote $R^{v}_f(G,p)=R^v(G,p)$ and $L^v_f(G,p)=L^v(G,p)$.

Let $G=(V,E)$ be a graph, and let $a:V\to\mathbb{N}$. 
Let
\[
    \blowup{V}{a}= \{(v,i):\, v\in V, 1\leq i\leq a(v)\}
\]
and
\[
    \blowup{E}{a}=\left\{ \{(u,i),(v,j)\} :\, (u,i),(v,j)\in \blowup{V}{a}, \{u,v\}\in E\right\}.
\]
We call the graph $\blowup{G}{a}=(\blowup{V}{a},\blowup{E}{a})$ the \emph{$a$-blow-up of $G$}.
For $p:V\to \Rea^d$, we define $\blowup{p}{a}:\blowup{V}{a}\to \Rea^d$ by
\[
    \blowup{p}{a}((v,i))=p(v)
\]
for all $(v,i)\in \blowup{V}{a}$.
In the special case when $a$ is the constant function  $a\equiv k$, for some $k\in \mathbb{N}$, we denote $\blowup{G}{a}=\blowup{G}{k}$ and $\blowup{p}{a}=\blowup{p}{k}$. 

For a symmetric matrix $M\in\Rea^{n\times n}$, let $\text{Spec}(M)$ be the \emph{spectrum} of $M$. That is, $\text{Spec}(M)$ is the multiset whose elements are the eigenvalues of $M$. For a multiset $S$ and $k\in\mathbb{N}$, let $S^{[k]}$ be the multiset whose elements are the elements of $S$, each repeated $k$ times (for convenience, we define $S^{[0]}=\emptyset$). For $\alpha\in \Rea$, we denote by $\alpha S$ the multiset obtained from $S$ by multiplying each of its elements by $\alpha$.

Our main result is the following theorem, which provides a complete description of the stiffness matrix spectrum of the blow-up of a framework.

\begin{theorem}\label{thm:blow_up}
    Let $G=(V,E)$ be a graph, and let $p:V\to \Rea^d$ and $a:V\to\mathbb{N}$.
    Then,
    \[
        \text{Spec}\left(L(\blowup{G}{a},\blowup{p}{a})\right)= \left(\bigcup_{v\in V} \text{Spec}(L_a^v(G,p))^{[a(v)-1]}\right)\cup \text{Spec}(L_a(G,p)).
    \]
\end{theorem}

In the special case $a\equiv k$, for some $k\in\mathbb{N}$, we obtain the following.
\begin{corollary}\label{cor:blow_up}
    Let $G=(V,E)$ be a graph, and let $p:V\to\Rea^d$.
    Then, for all $k\geq 2$, 
        \[
        \text{Spec}\left(L(\blowup{G}{k},\blowup{p}{k})\right)= k \left(\left(\bigcup_{v\in V} \text{Spec}(L^v(G,p))^{[k-1]}\right)\cup \text{Spec}( L(G,p))\right).
    \]
    In particular,
    \[
    \lambda_{\binom{d+1}{2}+1}\left(L(\blowup{G}{k},\blowup{p}{k})\right)
    =\frac{k}{2}\cdot  \lambda_{\binom{d+1}{2}+1}\left(L(\blowup{G}{2},\blowup{p}{2})\right).
    \]
\end{corollary}
Corollary \ref{cor:blow_up} solves a conjecture posed by Lew, Nevo, Peled and Raz in \cite[Conjecture 6.3]{LNPR2023onthe}, corresponding to the special case when $G=K_n$.

\subsection{The $d$-dimensional algebraic connectivity of complete bipartite graphs}

Let $K_{n,m}$ be the complete bipartite graph with sides of size $n$ and $m$, respectively. Bolker and Roth commenced in \cite{BR80} the study of the rigidity of complete bipartite graphs. Based on their work, Whiteley \cite{Whi84}, and independently Raymond \cite{Ray84} (see also \cite[Theorem 1.2]{Jor22} and \cite[Corollary 4]{Nix21}), proved the following.

\begin{theorem}[Whiteley \cite{Whi84}, Raymond \cite{Ray84}]\label{thm:whiteley}
Let $n,m\geq 2$. Then, the complete bipartite graph $K_{n,m}$ is $d$-rigid if and only if $n,m\geq d+1$ and $n+m\geq \binom{d+2}{2}$.
\end{theorem}

See also \cite[Proposition 1.6]{krivelevich2023rigid} for an alternative proof of the sufficiency part of Theorem \ref{thm:whiteley}.   
It was shown by Presenza, Mas, Giribet and Alvarez-Hamelin in \cite[Theorem 1]{presenza2022upper}, and independently by Lew, Nevo, Peled and Raz in \cite[Theorem 1.6]{LNPR23+}, that for every graph $G$ and every $d\geq 1$, $a_d(G)\leq a_1(G)$. For a complete bipartite graph, we obtain as a consequence $a_d(K_{n,m})\leq a_1(K_{n,m})=\min\{n,m\}$. As an application of Theorem \ref{thm:blow_up}, we obtain a lower bound on $a_d(K_{n,m})$, which matches this upper bound up to a multiplicative constant.

\begin{theorem}\label{thm:knm}
   Let $d\geq 1$. Then, there exists 
   $c_d>0$ such that, for all $n,m\geq d+1$ satisfying $n+m\ge \binom{d+2}{2}$,
   \[
        a_d(K_{n,m})\geq c_d \cdot \min\{n,m\}.
   \]
\end{theorem}

In general, we don't have explicit bounds on the constant $c_d$. However, for small values of $d$, we prove the following explicit (but non-sharp) bounds.

\begin{theorem}\label{thm:knm_small_d}
Let $n,m> 10$. Then,
 \[
        a_2(K_{n,m})\geq 0.278 \cdot \min\{n,m\} - 2.78,
    \]
    and
    \[
        a_3(K_{n,m})\geq 0.0618 \cdot \min\{n,m\} - 0.618.
    \]    
\end{theorem}

In the special case $d=2$, $n=m=3$, we obtain the following lower bound, which we conjecture to be optimal.

\begin{theorem}\label{thm:k33_tight}
    Let $\lambda\approx 0.6903845$ be the unique positive real root of the polynomial $176 x^4-200 x^3+47 x^2+18 x-9$. Then,
    \[
        a_2(K_{3,3})\ge 2(1-\lambda)\approx 0.6192309.
    \]
\end{theorem}
Theorem \ref{thm:k33_tight} follows by analyzing the limiting behavior of certain sequences of embeddings of the vertex set of $K_{3,3}$ in the plane (see Section \ref{sec:k33} for more details).

\begin{remark}
The application of graph blow-ups to the study of the $d$-dimensional algebraic connectivity of graphs was initiated by Jord\'an and Tanigawa in \cite{JT22}, and further developed by Lew, Nevo, Peled and Raz in \cite{LNPR2023onthe,LNPR23+}. In contrast to the ad-hoc arguments utilized in both \cite{JT22} and \cite{LNPR2023onthe,LNPR23+}, Theorem \ref{thm:blow_up} provides us with a general method for computing the stiffness eigenvalues of framework blow-ups. We expect this method to be useful in the future in the study of the high dimensional algebraic connectivity of other families of graphs as well.
\end{remark}

The paper is organized as follows. Section \ref{sec:prelims} contains  background material on matrix eigenvalues and some preliminary results on ``lower stiffness matrices" of frameworks. In Section \ref{sec:main_proof} we present the proofs of our main result, Theorem \ref{thm:blow_up}, and its Corollary \ref{cor:blow_up}. In Section \ref{sec:bounds} we obtain, as a consequence of Theorem \ref{thm:blow_up}, a lower bound on the $d$-dimensional algebraic connectivity of the $a$-blow-up of a graph (Lemma \ref{lemma:eigenvalue_bound}), which we will later use. In Section \ref{sec:completebipartitegraphs} we prove Theorems \ref{thm:knm}, \ref{thm:knm_small_d} and \ref{thm:k33_tight}, dealing with the $d$-dimensional algebraic connectivity of complete bipartite graphs. In Section \ref{sec:star} we present another application of Theorem \ref{thm:blow_up}, to the study of ``generalized star graphs". In particular, we obtain a new new, simpler proof of a result from \cite{LNPR23+} regarding the $d$-dimensional algebraic connectivity of these graphs. We close with some open problems and concluding remarks in Section \ref{sec:concluding_remarks}.

\begin{remark}[Notation and terminology.] Let $G=(V,E)$ be a graph. Occasionally, we will denote the vertex set of $G$ by $V(G)$ and its edge set by $E(G)$. 
For $n\in \mathbb{N}$, we denote by $[n]$ the set $\{1,2,\ldots,n\}$ and by $\binom{[n]}{2}$ the set $\{\{i,j\}:\, 1\le i<j\le n\}$. We denote the $n\times n$ identity matrix by $I_n$.
\end{remark}

\section{Preliminaries}\label{sec:prelims}

\subsection{Eigenvalues of symmetric matrices}

Let $M\in \Rea^{n\times n}$ be a symmetric matrix. For $1\le i\le n$, we denote by $\lambda_i(M)$ the $i$-th smallest eigenvalue of $M$. 
We will need the following well-known lemma (see, for example, \cite[Chapter 9]{marshall2011inequalities}).

\begin{lemma}\label{lemma:AB_BA_eigenvalues}
    Let $n, m \ge 1$, and let $A\in \Rea^{n\times m}$. Then, for all $1\le i\le m$,
    \[
        \lambda_i(A^{\T} A)= \begin{cases} 0 & \text{if } m\ge n \text{ and } 1\le i\le m-n,\\
        \lambda_{i+n-m}(A A^{\T}) & \text{otherwise.}
        \end{cases}
    \]
\end{lemma}
In other words, $A^{\T}A$ has the same non-zero eigenvalues as $AA^{\T}$ (with the same multiplicities), and the difference between the multiplicity of $0$ as an eigenvalue of $A^{\T}A$ and its multiplicity as an eigenvalue of $A A^{\T}$ is exactly $m-n$.

We will also need the following theorem of Ostrowski and its ensuing corollary.

\begin{theorem}[Ostrowski {\cite[Theorem 1]{ostrowski1959quantitative}}]\label{thm:ost}
  Let $H\in \Rea^{n\times n}$ be a symmetric matrix, and let $S\in \Rea^{n\times n}$ be a non-singular matrix. Then, for all $1\leq i\leq n$,
  \[
      \lambda_1(S S^{\T}) \cdot \lambda_i(H)\leq  \lambda_i(S H S^{\T})\leq \lambda_n(S S^{\T}) \cdot \lambda_i(H).
  \]
\end{theorem}

\begin{corollary}\label{cor:ost}
    Let $n, m \ge 1$ be integers. Let $A\in \Rea^{n\times m}$, and let $S\in \Rea^{n\times  n}$ be a non-singular matrix. Then, for $1\le i\le m$,
    \[
    \lambda_1(S S^{\T})\cdot  \lambda_i(A^{\T} A) \le \lambda_i(A^{\T} S^{\T} S A)\le \lambda_n(S S^{\T})\cdot  \lambda_i(A^{\T} A).
    \]
\end{corollary}
\begin{proof}
    Let $1\le i\le m$. If $m\ge n$ and $1\le i\le m-n$, then, by Lemma \ref{lemma:AB_BA_eigenvalues},
\[
     \lambda_i( A^{\T} S^{\T} S A)=0,
\]
and
\[
    \lambda_i(A^{\T} A)= 0.
\]
Therefore, the claim holds trivially in this case. Otherwise, by Lemma \ref{lemma:AB_BA_eigenvalues}, we have
    \[
        \lambda_i( A^{\T} S^{\T} S A)= \lambda_{i+n-m}(S A A^{\T} S^{\T}),
    \]
    and
    \[
    \lambda_i(A^{\T} A)= \lambda_{i+n-m}(A A^{\T}).
    \]
    Hence, applying Theorem \ref{thm:ost} to the matrix $H=A A^{\T}$, we obtain
   \[
    \lambda_1(S S^{\T})\cdot  \lambda_i(A^{\T} A) \le \lambda_i(A^{\T} S^{\T} S A)\le \lambda_n(S S^{\T})\cdot  \lambda_i(A^{\T} A),
   \]
   as wanted.
\end{proof}

\subsection{Lower stiffness matrices}
Let $G=(V,E)$ be a graph, and let $p: V\to \Rea^d$. The lower stiffness matrix $\lminusva{}{}(G,p)\in \Rea^{|E|\times |E|}$ is defined as 
\[
\lminusva{}{}(G,p) = R(G,p)^{\T}R(G,p).
\]
Similarly to the stiffness matrix, the lower stiffness matrix $\lminusva{}{}(G,p)$ is positive semi-definite, and satisfies $\rank(\lminusva{}{}(G,p)) = \rank(R(G,p))$. Note that, by Lemma \ref{lemma:AB_BA_eigenvalues}, the non-zero eigenvalues of $\lminusva{}{}(G,p)$, with multiplicities, coincide with those of $L(G,p)$  (see Lemma \ref{lemma:lower_stiffness_spectrum} below).

For $f: V \to \Rea_{>0}$, we define
\[\lminusva{}{f}(G,p)=R_f(G,p)^{\T} R_f(G,p)\in \Rea^{|E|\times |E|}\]
and, for $v\in V$, 
\[
\lminusva{v}{f}(G,p) = R^v_f(G,p)^{\T} R^v_f(G,p)\in \Rea^{\deg_G(v)\times \deg_G(v)}.
\]
If $f\equiv 1$, we denote $\lminusva{v}{f}(G,p)=\lminusva{v}{}(G,p)$.

The following lemmas give explicit descriptions of the entries of the matrices  $\lminusva{}{f}(G,p)$ and $\lminusva{v}{f}(G,p)$, respectively. For $e_1 = \{u,v\}$ and $e_2 = \{u,w\}$, we denote by $\theta_p(e_1, e_2)$ the angle between the vectors $d_{uv}$ and $d_{uw}$. For convenience, if one of the vectors $d_{uv}$ or $d_{uw}$ is the zero vector, we define $\cos(\theta_p(e_1,e_2))=0$.

\begin{lemma}\label{lemma:lower_weighted_stiffness}
Let $G=(V,E)$ be a graph, $p:V\to \Rea^d$ and $f:V\to \Rea_{>0}$. Let $e,e'\in E$. Then,
    \[
    L_f^{\downarrow}(G,p)_{e,e'}=\begin{cases}
                        f(u)+f(v) & \text{if } e=e'=\{u,v\} \text{ and } p(u) \ne p(v),\\
                        \sqrt{f(v) f(w)}\cos(\theta_p(e,e')) & \text{if } e=\{u,v\}, e'=\{u,w\},\\
                        0 & \text{otherwise.}
                    \end{cases}
\]
\end{lemma}

\begin{lemma}\label{lemma:lower_local_stiffness}
Let $G=(V,E)$ be a graph, $p:V\to \Rea^d$ and $f:V\to \Rea_{>0}$. Let $v\in V$, and $u,w\in N_G(v)$. Then,
    \[
    L_f^{v \downarrow}(G,p)_{u,w}=\begin{cases}
                        f(u) & \text{if } u=w \text{ and } p(u) \ne p(v),\\
                        0 &\text{if } u=w \text{ and } p(u)=p(v),\\
                        \sqrt{f(u) f(w)}\cos(\theta_p(\{v,u\},\{v,w\})) & \text{otherwise.}
                    \end{cases}
\]
\end{lemma}

\begin{proof}[Proof of Lemmas \ref{lemma:lower_weighted_stiffness} and \ref{lemma:lower_local_stiffness}]
Follows from direct computation, using the fact that $d_{uv}\cdot d_{uw} = \cos\theta$, where $\theta$ is the angle between $d_{uv}$ and $d_{uw}$.
\end{proof}

\begin{remark}
The unweighted case of Lemma \ref{lemma:lower_weighted_stiffness} was first stated and proved in \cite[Lemma 2.1]{LNPR2023onthe}.
\end{remark}

The following Lemma follows immediately from Lemma \ref{lemma:lower_weighted_stiffness} (applied to the constant function $f\equiv 1$) and the definition of an eigenvector.

\begin{lemma}[See {\cite[Lemma 4.5]{LNPR2023onthe}}]\label{lemma:lowerlaplacian}
Let $G=(V,E)$ be a graph and let $p:V\to \Rea^d$ such that $p(u)\ne p(v)$ for all $\{u,v\}\in E$. Then, a vector
$\psi\in \Rea^{|E|}$ is an eigenvector of $\lminusva{}{}(G,p)$ with eigenvalue $\lambda$ if and only if, for all $e\in E$,
\[
(\lambda - 2)\psi_e = \sum_{\substack{ e'\in E:\\ |e\cap e'| = 1}}\cos(\theta_p(e,e'))\psi_{e'}.
\]
\end{lemma}

As a direct application of Lemma \ref{lemma:AB_BA_eigenvalues}, we obtain the following relation between the eigenvalues of the stiffness and lower stiffness matrices.

\begin{lemma}\label{lemma:lower_stiffness_spectrum}
Let $G=(V,E)$ be a graph, and let $p:V\to \Rea^d$ and $f:V\to \Rea_{>0}$. Then, for all $1\le i\le |E|$,
    \[
        \lambda_i(\lminusva{}{f}(G,p))= \begin{cases} 0 & \text{if } |E|\ge d|V| \text{ and } 1\le i\le |E|-d|V|,\\
        \lambda_{i+d|V|-|E|}(L_f(G,p)) & \text{otherwise,}
        \end{cases}
    \]
    and, for all $v\in V$,
    \[
        \lambda_i(\lminusva{v}{f}(G,p))= \begin{cases} 0 & \text{if } \deg_G(v)\ge d \text{ and } 1\le i\le \deg_G(v)-d,\\
        \lambda_{i+d-\deg_G(v)}(L^v_f(G,p)) & \text{otherwise.}
        \end{cases}
    \]
\end{lemma}
\begin{proof}
    Since 
    \[
        \lminusva{}{f}(G,p) = R_f(G,p)^{\T}R_f(G,p) \in \Rea^{|E| \times |E|}
    \]
    and 
    \[
        L_f(G,p) = R_f(G,p)R_f(G,p)^{\T} \in \Rea^{d|V| \times d|V|},
    \]
    the first statement follows directly from Lemma \ref{lemma:AB_BA_eigenvalues}. 
    Similarly, since, for every $v\in V$, 
    \[
        \lminusva{v}{f}(G,p) = R^v_f(G,p)^{\T}R^v_f(G,p) \in \Rea^{\deg_G(v) \times \deg_G(v)}
    \]
    and 
    \[
        L^v_f(G,p) = R^v_f(G,p)R^v_f(G,p)^{\T} \in \Rea^{d \times d},
    \]
    the second statement also follows directly from Lemma \ref{lemma:AB_BA_eigenvalues}.
\end{proof}

Finally, let us state the following analog of Theorem \ref{thm:blow_up} for lower stiffness matrices.

\begin{theorem}\label{thm:lower_stiffness_blow_up}
  Let $G=(V,E)$ be a graph, $p:V\to \Rea^d$ and $a:V\to\mathbb{N}$.
    Then,
    \[
        \text{Spec}\left(\lminusva{}{}(\blowup{G}{a},\blowup{p}{a})\right) = \left(\bigcup_{v\in V} \text{Spec}(\lminusva{v}{a}(G,p))^{[a(v)-1]}\right)\cup \text{Spec}(\lminusva{}{a}(G,p))\cup \{0\}^{[N]},
    \]
    where 
    \[
    N = \sum_{\{u,v\}\in E} (a(u)-1)(a(v)-1).
    \]
\end{theorem}
Theorem \ref{thm:lower_stiffness_blow_up} follows easily from Theorem \ref{thm:blow_up} and Lemma \ref{lemma:lower_stiffness_spectrum}, as detailed next.
\begin{proof}[Proof of Theorem \ref{thm:lower_stiffness_blow_up}]
    For a symmetric matrix $M$, let $\Spec_{\ne 0}(M)$ be the multiset consisting of the nonzero eigenvalues of $M$, and let $\mult_0(M)$ be the multiplicity of $0$ as an eigenvalue of $M$ (in other words, the dimension of the kernel of $M$). 
    By Theorem \ref{thm:blow_up}, we have that
    \[
    \text{Spec}\left(L(\blowup{G}{a},\blowup{p}{a})\right)= \left(\bigcup_{v\in V} \text{Spec}(L_a^v(G,p))^{[a(v)-1]}\right)\cup \text{Spec}(L_a(G,p)).
    \]
    Hence, we have 
    \begin{equation}\label{eq:non_zero_spec}
    \Spec_{\ne 0}\left(L(\blowup{G}{a},\blowup{p}{a})\right)= \left(\bigcup_{v\in V} \Spec_{\ne 0}(L_a^v(G,p))^{[a(v)-1]}\right)\cup \Spec_{\ne 0}(L_a(G,p))
    \end{equation}
    and
    \begin{equation}\label{eq:mult0}
    \mult_0\left(L(\blowup{G}{a},\blowup{p}{a})\right)= \left(\sum_{v\in V} \mult_0(L_a^v(G,p))(a(v)-1)\right) + \mult_0(L_a(G,p)).
    \end{equation}
    It follows from Lemma \ref{lemma:lower_stiffness_spectrum} that 
    \[
    \Spec_{\ne 0}\left(\lminusva{}{}(\blowup{G}{a},\blowup{p}{a})\right)=\Spec_{\ne 0}\left(L(\blowup{G}{a},\blowup{p}{a})\right),
    \]
    \[
    \Spec_{\ne 0}(\lminusva{}{a}(G,p)) = \Spec_{\ne 0}(L_a(G,p)),
    \]
    and, for all $v\in V$,
     \[
    \Spec_{\ne 0}(\lminusva{v}{a}(G,p)) = \Spec_{\ne 0}(L_a^v(G,p)).
    \]
    Similarly, again by Lemma \ref{lemma:lower_stiffness_spectrum},
    \[
    \mult_0\left(\lminusva{}{}(\blowup{G}{a},\blowup{p}{a})\right)=\mult_0\left(L(\blowup{G}{a},\blowup{p}{a})\right) + |\blowup{E}{a}|-d|
    \blowup{V}{a}|,
    \]
    \[
    \mult_0(\lminusva{}{a}(G,p)) = \mult_0(L_a(G,p)) + |E|-d|V|,
    \]
    and, for all $v\in V$, 
     \[
    \mult_0(\lminusva{v}{a}(G,p)) = \mult_0(L_a^v(G,p)) + \deg_G(v)-d.
    \]
Hence, by \eqref{eq:non_zero_spec}, we have
\[
\Spec_{\ne 0}\left(\lminusva{}{}(\blowup{G}{a},\blowup{p}{a})\right)= \left(\bigcup_{v\in V} \Spec_{\ne 0}(\lminusva{v}{a}(G,p))^{[a(v)-1]}\right)\cup \Spec_{\ne 0}(\lminusva{}{a}(G,p)).
\]  
Moreover, by \eqref{eq:mult0},
    \[
    \mult_0\left(\lminusva{}{}(\blowup{G}{a},\blowup{p}{a})\right) - N_1 = \left(\sum_{v\in V} \mult_0(\lminusva{v}{a}(G,p))(a(v)-1)\right)+ \mult_0(\lminusva{}{a}(G,p)) - N_2 - N_3, 
    \]
    where $N_1 = |\blowup{E}{a}|-d|
    \blowup{V}{a}|$, $N_2 = \sum_{v\in V}(\deg_G(v)-d)(a(v)-1)$, and $N_3 = |E|-d|V|$. Therefore, it is sufficient to show $N_1-N_2-N_3 = N$. Notice that by double counting, we get 
    \begin{align*}
        N_2 &= \sum_{v\in V}(\deg_G(v)-d)(a(v)-1)\\
        &= \sum_{v\in V}a(v)\deg_G(v) - d\sum_{v\in V}a(v) - \sum_{v\in V}\deg_G(v) + \sum_{v\in V}d\\
        &= \sum_{\{u,v\}\in E}(a(u)+a(v)) - d|\blowup{V}{a}| - 2|E| + d|V|\\
        &= \sum_{\{u,v\}\in E}(a(u)+a(v)-1) - d|\blowup{V}{a}| - |E| + d|V|.
    \end{align*}
    Hence, 
    \begin{align*}
        N_1-N_2-N_3 &= |\blowup{E}{a}| - \sum_{\{u,v\}\in E}(a(u)+a(v)-1)\\
        &= \sum_{\{u,v\}\in E}a(u)a(v) - \sum_{\{u,v\}\in E}(a(u)+a(v)-1)\\
        &= \sum_{\{u,v\}\in E}(a(u)a(v)-a(v)-a(u)+1)\\
        &= \sum_{\{u,v\}\in E}(a(u)-1)(a(v)-1) = N,
    \end{align*}
    as required.
\end{proof}

\section{Spectra of stiffness matrices of graph blow-ups}\label{sec:main_proof}

In this section we present the proofs of Theorem \ref{thm:blow_up} and Corollary \ref{cor:blow_up}. For convenience, for a vertex set $V$,  we identify the vectors in $\Rea^{d|V|}$ with functions from $V$ to $\Rea^d$. Namely, for $\phi\in \Rea^{d|V|}$ and $u\in V$, we denote by $\phi(u)\in \Rea^d$ the vector consisting of the $d$ coordinates of $\phi$ indexed by the vertex $u$.

Let $U$ be a finite vector space, and let $M:U\to U$ be a linear operator. For a subspace $U'\subset U$, denote by $M\big|_{U'}$ the restriction of $M$ to $U'$. If $U'$ is an invariant subspace of $M$, we consider $M\big|_{U'}$ as an operator from $U'$ to itself.

For a graph $G=(V,E)$ and $u,v\in V$, we write $u\sim v$ if $\{u,v\}\in E$, and $u\not\sim v$ otherwise. 
We will need the following Lemma.

\begin{lemma}
    \label{lemma:L_formula}
    Let $G=(V,E)$ be a graph, and let $p:V\to\Rea^d$ and $f:V\to \Rea_{>0}$. Then, for all $\phi\in \Rea^{d|V|}$ and $u\in V$,
    \[
        L_f(G,p)\phi(u)=\sum_{v\sim u} d_{uv} d_{uv}^{\T} \left(f(v) \phi(u)- \sqrt{f(v)f(u)}\phi(v)\right).
    \]
\end{lemma}
\begin{proof}
    Let $R_f=R_f(G,p)$ and $L_f=L_f(G,p)$. 
    By the definition of $R_f$, we have, for $e=\{u,v\}\in E$,
    \[
       (R_f^{\T} \phi)_e= d_{uv}^{\T} \left(\sqrt{f(v)}\phi(u)-\sqrt{f(u)}\phi(v)\right).
    \]
    Then, again by the definition of $R_f$, we obtain for all $u\in V$,
    \[
    L_f \phi(u)= (R_f R_f^{\T}\phi)(u)= \sum_{v\sim u} \sqrt{f(v)} d_{uv}(R_f^{\T}\phi)_{\{u,v\}}
    =\sum_{v\sim u} d_{uv} d_{uv}^{\T} \left(f(v) \phi(u)- \sqrt{f(v)f(u)}\phi(v)\right),
    \]
    as wanted.
\end{proof}

\begin{proof}[Proof of Theorem \ref{thm:blow_up}]

We define a linear transformation $T: \Rea^{d|V|}\to \Rea^{d |\blowup{V}{a}|}$ by
\[
    T\phi((u,i))= a(u)^{-1/2}\phi(u)
\]
for every $\phi\in \Rea^{d|V|}$ and $(u,i)\in \blowup{V}{a}$. 
Similarly, for every $v\in V$ with $a(v)\geq 2$ and every $2\leq k\leq a(v)$, define $T_{v,k}:\Rea^d\to \Rea^{d|\blowup{V}{a}|}$ by
\[
    T_{v,k}x ((u,i))= \begin{cases}
         x & \text{if } u=v,\, i=1,\\

        -x & \text{if } u=v,\, i=k,\\
        0 & \text{otherwise,}
    \end{cases}
\]
for every $x\in \Rea^d$ and $(u,i)\in \blowup{V}{a}$.

\begin{claim}\label{claim:1}
\[
    \Rea^{d|\blowup{V}{a}|} = \left(\bigoplus_{\substack{v\in V,\\ 2\leq k\leq a(v)}} \Ima T_{v,k}\right)\bigoplus \Ima T.
    \]
\end{claim}
\begin{proof}

    First, it is easy to check that $\Ima T \perp \Ima T_{v,k}$ for all $v\in V$ and $2\leq k\leq a(v)$. Moreover, note that the sum of the subspaces $\Ima T_{v,k}$, for $v\in V$ and $2\leq k\leq a(v)$, is a direct sum, since for every $v\in V$ and $2\leq k\leq a(v)$ there is a coordinate in $\Rea^{d|\blowup{V}{a}|}$ that is unique to the support of the vectors in $\Ima T_{v,k}$ (in fact, the $d$ coordinates associated with the vertex $(v,k)$ are unique to the support of vectors in $\Ima T_{v,k}$). 
    Therefore, we have a direct sum 
    \[
    \left(\bigoplus_{\substack{v\in V,\\ 2\leq k\leq a(v)}} \Ima T_{v,k}\right)\bigoplus \Ima T\subset \Rea^{d|\blowup{V}{a}|}.
    \]

    Finally, note that $T$ and $T_{v,k}$, for $v\in V$ and $2\leq k\leq a(v)$, are injective linear transformations, and therefore $\dim(\Ima T)= d|V|$ and $\dim(\Ima T_{v,k})=d$. Hence,
    \begin{align*}
        \dim\left (\left(\bigoplus_{\substack{v\in V,\\ 2\leq k\leq a(v)}} \Ima T_{v,k}\right)\bigoplus \Ima T\right)
        &= d|V|+\sum_{v\in V}(a(v)-1)d
        \\
        &= \left(\sum_{v\in V} a(v)\right)d= d|\blowup{V}{a}|= \dim(\Rea^{d|\blowup{V}{a}|}).
    \end{align*}
    So $\left(\bigoplus_{\substack{v\in V,\\ 2\leq k\leq a(v)}} \Ima T_{v,k}\right)\bigoplus \Ima T= \Rea^{d|\blowup{V}{a}|}$, as wanted.
\end{proof}

We denote $R=R(\blowup{G}{a},\blowup{p}{a})$ and $L=L(\blowup{G}{a},\blowup{p}{a})$.

\begin{claim}\label{claim:2}
The subspace $\Ima T$ is invariant under $L$. Moreover, 
\[
    \text{Spec}\left(L\big|_{\Ima T}\right)=\text{Spec}\left(L_a(G,p)\right).
    \]
\end{claim}
\begin{proof}

Let $\phi\in \Rea^{d|V|}$. Applying Lemma \ref{lemma:L_formula} to $\left(\blowup{G}{a},\blowup{p}{a}\right)$ (with respect to the constant function $f\equiv 1$), we obtain, for all $(u,i)\in \blowup{V}{a}$,
\begin{align*}
    L T \phi((u,i))&= \sum_{(v,j)\sim (u,i)} d_{uv}d_{uv}^{\T}\left(
    T\phi((u,i))-T\phi((v,j)) \right)
   \\ &= 
     \sum_{(v,j)\sim (u,i)} d_{uv}d_{uv}^{\T}\left(
    a(u)^{-1/2}\phi(u)-a(v)^{-1/2}\phi(v) \right)
 \\ &=\sum_{v\sim u} d_{uv} d_{uv}^{\T} \left(a(u)^{-1/2}a(v)\phi(u)-a(v)^{1/2}\phi(v)\right).
\end{align*}
Now, let $L_a=L_a(G,p)$. By Lemma \ref{lemma:L_formula}, we have, for all $u\in V$,
\[
 L_a \phi (u) =\sum_{v\sim u} d_{uv} d_{uv}^{\T} \left( a(v)\phi(u)-(a(u)a(v))^{1/2}\phi(v)\right).
\]
So, for $(u,i)\in \blowup{V}{a}$,
\[
T L_a \phi ((u,i)) = \sum_{v\sim u} d_{uv} d_{uv}^{\T} \left( a(u)^{-1/2}a(v)\phi(u)-a(v)^{1/2}\phi(v)\right)= L T \phi ((u,i)).
\]
Hence, $T L_a = L T$. That is, $\Ima T$ is invariant under $L$, and $L\big|_{\Ima T}$ is similar to $L_a$. Therefore, the two operators have the same spectrum.

\end{proof}

\begin{claim}\label{claim:3}
Let $v\in V$ and $2\leq k\leq a(v)$. Then, the subspace $\Ima T_{v,k}$ is invariant under $L$. Moreover, 
\[
    \text{Spec}\left(L\big|_{\Ima T_{v,k}}\right)=\text{Spec}\left(L_a^v(G,p)\right).
    \]
\end{claim}
\begin{proof}
Let $v\in V$,  $2\leq k\leq a(v)$, and $x\in \Rea^d$. By Lemma \ref{lemma:L_formula}, we have, for all $(u,i)\in \blowup{V}{a}$,
\begin{equation}\label{eq:last_claim_formula}
    L T_{v,k}x ((u,i))= \sum_{(w,j)\sim (u,i)} d_{uw} d_{uw}^{\T} \left( T_{v,k}x((u,i))- T_{v,k}x((w,j))\right).
\end{equation}
First, assume that $u\ne v$ or $i\notin\{1,k\}$. Then, $T_{v,k}x((u,i))=0$, and so we have
\[
    L T_{v,k}x ((u,i))= -\sum_{(w,j)\sim (u,i)} d_{uw} d_{uw}^{\T} T_{v,k}x((w,j)).
\]
If $u\not\sim v$, we have $T_{v,k}x((w,j))=0$ for all $(w,j)\sim (u,i)$, and so $L T_{v,k}x ((u,i))=0$. 
Similarly, if $u\sim v$, we obtain
\[
    L T_{v,k}x ((u,i))= - d_{uv} d_{uv}^{\T} T_{v,k}x((v,1)) - d_{uv} d_{uv}^{\T} T_{v,k}x((v,k)) =
    -d_{uv} d_{uv}^{\T} x +d_{uv} d_{uv}^{\T} x=0.
\]
Now, assume $u=v$ and $i\in\{1,k\}$. Note that $T_{v,k}x((w,j))=0$ for all $(w,j)\sim (v,i)$. So, if $i=1$, we obtain, by \eqref{eq:last_claim_formula},
\[
L T_{v,k} x((v,1))= \sum_{(w,j)\sim (v,1)} d_{vw} d_{vw}^{\T}  T_{v,k}x((v,1)) =\sum_{(w,j)\sim (v,1)} d_{vw} d_{vw}^{\T} x = \sum_{w\sim v} a(w) d_{vw} d_{vw}^{\T} x.
\]
Similarly, if $i=k$,
\[
L T_{v,k} x((v,k))= \sum_{(w,j)\sim (v,k)} d_{vw} d_{vw}^{\T}  T_{v,k}x((v,k)) =-\sum_{(w,j)\sim (v,k)} d_{vw} d_{vw}^{\T} x = -\sum_{w\sim v} a(w) d_{vw} d_{vw}^{\T} x.
\]
In conclusion,
\[
  L T_{v,k} x ((u,i)) =\begin{cases}
      \sum_{w\sim v} a(w) d_{vw}d_{vw}^{\T} x & \text{if } u=v \text{ and } i=1,\\
      -\sum_{w\sim v} a(w) d_{vw}d_{vw}^{\T} x & \text{if } u=v \text{ and } i=k,\\
      0 & \text{otherwise.}
  \end{cases}
\]
Now, let $L_a^v=L_a^v(G,p)$. Then,
\[L_a^v=R_a^v(G,p) R_a^v(G,p)^{\T}=\sum_{w\sim v} a(w)d_{vw} d_{vw}^{\T}.
\]
Hence, $T_{v,k} L_a^v =L T_{v,k}$.  That is, $\Ima T_{v,k}$ is invariant under $L$, and $L\big|_{\Ima T_{v,k}}$ is similar to  $L_a^v$. Therefore, the two operators have the same spectrum.

\end{proof}

Finally, by Claims \ref{claim:1}, \ref{claim:2} and \ref{claim:3}, we obtain
 \[
        \text{Spec}\left(L(\blowup{G}{a},\blowup{p}{a})\right)= \left(\bigcup_{v\in V} \text{Spec}(L_a^v(G,p))^{[a(v)-1]}\right)\cup \text{Spec}(L_a(G,p)).
    \]

\end{proof}

\begin{proof}[Proof of Corollary \ref{cor:blow_up}]

Let $a:V\to\mathbb{N}$ be defined by $a(v)=k$ for all $v\in V$. It is easy to check that $L_a(G,p)=k L(G,p)$, and $L_a^v(G,p)=k L^v(G,p)$ for all $v\in V$. 
Hence, by Theorem \ref{thm:blow_up}, 
    \begin{align*}
        \text{Spec}\left(L(\blowup{G}{k},\blowup{p}{k})\right) &= \left(\bigcup_{v\in V} \text{Spec}(L_a^v(G,p))^{[a(v)-1]}\right)\cup \text{Spec}(L_a(G,p))\\
        &= \left(\bigcup_{v\in V} \text{Spec}(kL^v(G,p))^{[k-1]}\right)\cup \text{Spec}(kL(G,p))\\
        &= k \left(\left(\bigcup_{v\in V} \text{Spec}(L^v(G,p))^{[k-1]}\right)\cup \text{Spec}( L(G,p))\right).
    \end{align*}
    Since $L(G,p)$ and $L^v(G,p)$ (for all $v\in V$) are positive semi-definite matrices, and the multiplicity of $0$ as an eigenvalue of $L(G,p)$ is at least $\binom{d+1}{2}$, we obtain
\[
        \lambda_{\binom{d+1}{2}+1}\left(L(\blowup{G}{k},\blowup{p}{k})\right) = k\cdot \min\left\{
        \left\{\lambda_1(L^v(G,p)): v\in V\right\}
        \cup \left\{ 
        \lambda_{\binom{d+1}{2}+1}(L(G,p))\right\}
        \right\}.
\]
In particular, for $k=2$,
\[
        \lambda_{\binom{d+1}{2}+1}\left(L(\blowup{G}{2},\blowup{p}{2})\right) = 2\cdot \min\left\{
        \left\{\lambda_1(L^v(G,p)): v\in V\right\}
        \cup \left\{ 
        \lambda_{\binom{d+1}{2}+1}(L(G,p))\right\}
        \right\}.
\]
Hence, for all $k\ge 2$,
\[
        \lambda_{\binom{d+1}{2}+1}\left(L(\blowup{G}{k},\blowup{p}{k})\right) = \frac{k}{2}\cdot \lambda_{\binom{d+1}{2}+1}\left(L(\blowup{G}{2},\blowup{p}{2})\right),
\]
as wanted.

\end{proof}

\section{Eigenvalue bounds}\label{sec:bounds}

In this section we prove the following lemma, which we will later need for the proofs of Theorems \ref{thm:knm} and \ref{thm:knm_small_d}.

\begin{lemma}\label{lemma:eigenvalue_bound}
     Let $d\geq 1$, and let $G=(V,E)$ be a graph. Then, for all $a:V\to\mathbb{N}$,
     \[
     a_d(\blowup{G}{a}) \ge \sup_{p:V\to\Rea^d} \left(\min \left\{ h(a)\cdot \lambda_{\binom{d+1}{2}+1}(L(G,p)), \, g(a)\cdot \min_{v\in V} \lambda_1(L^v(G,p))\right\}\right),
     \]
     where
     \[
     h(a)= \frac{\min_{\{u,v\}\in E} a(u)\cdot a(v)}{\max_{w\in V} a(w)}
     \]
     and
     \[
        g(a)= \min_{v\in V} a(v).
     \]
\end{lemma}

We will need the following two simple lemmas, relating the weighted rigidity matrices of a framework to their unweighted counterparts.

\begin{lemma}\label{lemma:weighted_rigidity_matrix_form}
Let $G=(V,E)$ be a graph, and let $f:V\to \Rea_{>0}$. Then,
\[
    R_f(G,p)= \hat{D}_f R(G,p) \tilde{D}_f, 
\]
where $\hat{D}_f\in \Rea^{d|V|\times d|V|}$ is a diagonal matrix with entries
\[
    (\hat{D}_f)_{(v,i),(v,i)}= \frac{1}{\sqrt{f(v)}}
\]
for all $v\in V$ and $1\leq i\leq d$,
and $\tilde{D}_f\in \Rea^{|E|\times |E|}$ is a diagonal matrix with entries
\[
    (\tilde{D}_f)_{e,e}= \sqrt{f(u)f(v)}
\]
for all $e=\{u,v\}\in E$.
\end{lemma}
\begin{proof}
    Multiplying $R(G,p)$ on the left by $\hat{D}_f$ and on the right by $\tilde{D}_f$ results in multiplying the entry in row $(u,i)$, for $u\in V$ and $1\leq i\leq d$, and column $e=\{u,v\}\in E$, by
    \[
        \frac{\sqrt{f(u)f(v)}}{\sqrt{f(u)}}=\sqrt{f(v)},
    \]
    just as in the definition of $R_f(G,p)$.
\end{proof}

\begin{lemma}\label{lemma:weighted_rigidity_matrix_form_2}
Let $G=(V,E)$ be a graph, and let $f:V\to \Rea_{>0}$. Then, 
\[
    R_f^v(G,p)= R^v(G,p) D_f^v, 
\]
where $D_f^v \in \Rea^{\deg_G(v)\times \deg_G(v)}$ is a diagonal matrix with entries
\[
    (D_f^v)_{u,u} = \sqrt{f(u)}
\]
for all $u\in N_G(v)$.
\end{lemma}
   
\begin{proof}
     Right multiplication of $R^v(G,p)$ by $D_f^v$ results in multiplying the entries in column $u\in N_G(v)$ by
    $
        \sqrt{f(u)},
    $ 
    just as in the definition of $R^v_f(G,p)$.
\end{proof}

\begin{lemma}\label{lemma:scaling1}
    Let $G=(V,E)$ be a graph, and let $p:V\to\Rea^d$ and $f:V\to \mathbb{R}_{>0}$.
    Let
    \[
        C=\frac{\max\{ f(u)\cdot f(v) : \, \{u,v\}\in E\}}{\min\{f(v):\, v\in V\}},
    \]
    and
    \[
        c=\frac{\min\{ f(u)\cdot f(v) : \, \{u,v\}\in E\}}{\max\{f(v):\, v\in V\}}.
    \]   
Then, for all $1\leq k\leq d|V|$,
    \[
        c\cdot \lambda_k(L(G,p))\leq \lambda_{k}(L_f(G,p)) \leq C\cdot \lambda_k(L(G,p)).
    \]
\end{lemma}
\begin{proof}
Let $R=R(G,p)$, $L=L(G,p)$ and $L_f=L_f(G,p)$. Let $1\leq k\leq d|V|$. By Lemma \ref{lemma:weighted_rigidity_matrix_form}, we have
\[
  L_f= \hat{D}_f R \tilde{D}_f \tilde{D}_f^{\T} R^{\T} \hat{D}_f^{\T}.
\]
The eigenvalues of $\hat{D}_f \hat{D}_f^{\T}$ are $1/f(v)$ for $v\in V$ (each repeated $d$ times).
Therefore, by Theorem \ref{thm:ost},
\[
    \lambda_k(L_f)= c_1\cdot  \lambda_k(R \tilde{D}_f \tilde{D}_f^{\T} R^{\T}),
\]
where 
\[
\frac{1}{\max_{v\in V}{f(v)}} \leq c_1 \leq \frac{1}{\min_{v\in V}{f(v)}}.
\]
The eigenvalues of $\tilde{D}_f^{\T} \tilde{D}_f$ are $f(u)\cdot f(v)$, for $\{u,v\}\in E$. Therefore, by Corollary \ref{cor:ost},
\[
   \lambda_k(R \tilde{D}_f \tilde{D}_f^{\T} R^{\T})= c_2\cdot \lambda_k(R R^{\T})=c_2\cdot \lambda_k(L),
\]
where
\[
    \min_{\{u,v\}\in E} f(u)\cdot f(v)\le c_2\le \max_{\{u,v\}\in E} f(u)\cdot f(v).
\]
We obtain
\[
    \lambda_k(L_f)= c_1\cdot c_2\cdot \lambda_k(L),
\]
where
\[
              c\le c_1\cdot c_2\le C,
\]
as wanted.
\end{proof}

\begin{lemma}\label{lemma:scaling2}
    Let $G=(V,E)$ be a graph, and let $f:V\to \mathbb{R}_{>0}$. Let $v\in V$, and let
    \[
        C=\max\{ f(u) :\, u\in N_G(v)\}
    \]
    and
    \[
        c=\min\{ f(u) :\, u\in N_G(v)\}.
    \]   
Then, for all $1\leq k\leq d$,
    \[
        c\cdot \lambda_k(L^v(G,p))\leq \lambda_{k}(L_f^v(G,p)) \leq C\cdot \lambda_k(L^v(G,p)).
    \] 
\end{lemma}
\begin{proof}
    Let $R^v = R^v(G,p)$, $L^v = L^v(G,p)$ and $L_f^v = L_f^v(G,p)$. Let $1\le k \le d$. By Lemma \ref{lemma:weighted_rigidity_matrix_form_2}, we have
    \[
        L^v_f = R^v D_f^v {D_f^v}^{\T} {R^v}^{\T}.
    \]
    The eigenvalues of ${D_f^v}^{\T} {D_f^v}$ are $f(u)$, for $u\in N_G(v)$. Therefore, by Corollary \ref{cor:ost},
    \[
        \lambda_k (L^v_f) = \lambda_k(R^v D_f^v {D_f^v}^{\T} {R^v}^{\T})= c_1\cdot \lambda_k(R^v {R^v}^{\T})=c_1\cdot \lambda_k(L^v),
    \]
    where 
    \[
       c= \min_{u\in N_G(v)} f(u)\le c_1\le \max_{u\in N_G(v)} f(u)= C.
       \]
\end{proof}

\begin{proof}[Proof of Lemma \ref{lemma:eigenvalue_bound}]

By Theorem \ref{thm:blow_up}, we have that 
    \[
        \text{Spec}\left(L(\blowup{G}{a},\blowup{p}{a})\right)= \left(\bigcup_{v\in V} \text{Spec}(L_a^v(G,p))^{[a(v)-1]}\right)\cup \text{Spec}(L_a(G,p)).
    \]
Note that the multiplicity of $0$ as an eigenvalue of $L_a(G,p)$ is at least $\binom{d+1}{2}$ (since, by Lemma \ref{lemma:weighted_rigidity_matrix_form}, the rank of $L_a(G,p)$ is the same as the rank as $L(G,p)$). Therefore, since all eigenvalues of $L(\blowup{G}{a},\blowup{p}{a})$ are non-negative, we obtain 
\[
    \lambda_{\binom{d+1}{2}+1}\left(L(\blowup{G}{a},\blowup{p}{a})\right) = \min\left\{\lambda_{\binom{d+1}{2}+1}(L_a(G,p)),\, \min_{v\in V}\lambda_1(L_a^v(G,p))\right\}.
\]
By Lemma \ref{lemma:scaling1} and Lemma \ref{lemma:scaling2}, we have
\[
    \lambda_{\binom{d+1}{2}+1}(L_a(G,p)) \ge h(a)\cdot \lambda_{\binom{d+1}{2}+1}(L(G,p)),
\] 
and, for all $v\in V$, 
\[
    \lambda_1(L_a^v(G,p)) \ge g(a) \cdot \lambda_1(L^v(G,p)).
\]
Hence,
\[
    \lambda_{\binom{d+1}{2}+1}(L(\blowup{G}{a},\blowup{p}{a})) \ge \min \left\{ h(a)\cdot \lambda_{\binom{d+1}{2}+1}(L(G,p)), \, g(a)\cdot \min_{v\in V} \lambda_1(L^v(G,p))\right\}.
\]
Since 
\[
    a_d(\blowup{G}{a}) = \sup_{p^*:\blowup{V}{a}\to\Rea^d}\left(\lambda_{\binom{d+1}{2}+1}(L(\blowup{G}{a},p^*))\right) \ge \sup_{p:V\to\Rea^d}\left(\lambda_{\binom{d+1}{2}+1}(L(\blowup{G}{a},\blowup{p}{a}))\right),
\]
we conclude that
\[
    a_d(\blowup{G}{a}) \ge \sup_{p:V\to\Rea^d} \left(\min \left\{ h(a)\cdot \lambda_{\binom{d+1}{2}+1}(L(G,p)), \, g(a)\cdot \min_{v\in V} \lambda_1(L^v(G,p))\right\}\right).
\]
\end{proof}

\section{Complete bipartite graphs}\label{sec:completebipartitegraphs}

We proceed to prove Theorem \ref{thm:knm}. We will need the following Lemma.

\begin{lemma}\label{lemma:nonsingular_local_matrix}
    Let $G=(V,E)$ be a graph, let $p:V\to \Rea^d$ be a generic embedding, and let $f:V\to\Rea_{>0}$. Then, for all $v\in V$ with $\deg_G(v)\ge d$, $L^{v}_f(G,p)$ is non-singular.
\end{lemma}
\begin{proof}
Since $p$ is an embedding, we have $p(u)\ne p(v)$, and therefore $d_{vu}=(p(v)-p(u))/\|p(v)-p(u)\|$, for all $u\ne v$. Since $p$ is generic, the column vectors $\{\sqrt{f(u)}d_{vu}\}_{u\in N_G(v)}$ are in general position, that is, every $d$ of them are linearly independent. Finally, since $\deg_G(v)\ge d$, we obtain $\rank(R_f^v(G,p))=d$. That is, $R_f^v(G,p)$ has full row rank. Hence, $L_f^v(G,p)= R_f^v(G,p) R_f^v(G,p)^{\T}$ is non-singular.
\end{proof}

\begin{proof}[Proof of Theorem \ref{thm:knm}]

For $d=1$ the claim is trivial, so we assume $d\ge 2$.  
Let 
\[
S_d= \left\{ (n_0,m_0):\, n_0,m_0\ge d+1, \, n_0+m_0=\binom{d+2}{2}\right\}.
\]
For a graph $G=(V,E)$, let
\[
    b_d(G)= \sup_{p: V\to \Rea^d} \left( \min \left\{ \lambda_{\binom{d+1}{2}+1}(L(G,p)), \, \min_{v\in V} \lambda_1(L^v(G,p))\right\}\right).
\]
Let $(n_0,m_0)\in S_d$, and let $p:V(K_{n_0,m_0})\to \Rea^d$ be a generic embedding.
By Theorem \ref{thm:whiteley}, $K_{n_0,m_0}$ is $d$-rigid, and therefore $\lambda_{\binom{d+1}{2}+1}(L(K_{n_0,m_0},p))>0$. By Lemma \ref{lemma:nonsingular_local_matrix}, $\lambda_1(L^v(K_{n_0,m_0},p))>0$ for all $v\in V$. Therefore, $b_d(K_{n_0,m_0})>0$. Let
\[
c_d= \min \left\{ \frac{b_d(K_{n_0,m_0})}{4\cdot \max\{n_0,m_0\}} :\, (n_0,m_0)\in S_d\right\}.
\]
Note that $c_d>0$.

Now, let $n,m\ge d+1$ satisfying $n+m\ge \binom{d+2}{2}$. Let $(n_0,m_0)\in S_d$ such that $n\ge n_0$ and $m\ge m_0$. Let $q_1=\lfrac{n}{n_0}$ and $q_2=\lfrac{m}{m_0}$, and write $n=q_1 n_0+r_1$, $m=q_2 m_0+r_2$, where $0\le r_1<n_0$ and $0\le r_2<m_0$. Denote the partite sets of $K_{n_0,m_0}$ by $A=\{1,\ldots,n_0\}$ and $B=\{n_0+1,\ldots, n_0+m_0\}$. Define $a:A\cup B\to \mathbb{N}$ by
\[
    a(v)=\begin{cases}
        q_1+1 & \text{if } 1\le v\le r_1,\\
        q_1 & \text{if } r_1+1\le v\le n_0,\\
        q_2+1 & \text{if } n_0+1\le v\le n_0+r_2,\\
        q_2 & \text{if } n_0+r_2+1\le v\le n_0+m_0,
    \end{cases}
\]
for all $v\in A\cup B$. Note that $K_{n_0,m_0}^{(a)}=K_{n,m}$. Let
\begin{align*}
    h(a) &= \frac{ \min_{\{u,v\}\in E(K_{n_0,m_0})} a(u)\cdot a(v)}{\max_{w\in V(K_{n_0,m_0})} a(w)}\ge \frac{q_1\cdot q_2}{\max\{q_1+1,q_2+1\}}
    \\ &
    \ge \frac{q_1\cdot q_2}{2\cdot \max\{q_1,q_2\}}
    = \frac{1}{2}\cdot \min\{q_1,q_2\} \ge \frac{1}{4}\cdot \min\left\{\frac{n}{n_0},\frac{m}{m_0}\right\},
\end{align*}
and
\[
    g(a)= \min_{v\in V(K_{n_0,m_0})} a(v) = \min\{q_1,q_2\}\ge \frac{1}{2}\cdot \min\left\{\frac{n}{n_0},\frac{m}{m_0}\right\}.
\]
By Lemma \ref{lemma:eigenvalue_bound}, we obtain
\[
    a_d(K_{n,m})\ge b_d(K_{n_0,m_0}) \cdot \frac{1}{4} \cdot \min\left\{\frac{n}{n_0},\frac{m}{m_0}\right\} \ge c_d \cdot \min\{n,m\},
\]
as wanted.
\end{proof}

We proceed to prove Theorem \ref{thm:knm_small_d}. The proof is similar to the proof of Theorem \ref{thm:knm}, except that we analyze specific embeddings of a small complete bipartite graph in order to obtain explicit bounds for $d=2$ and $d=3$.

\begin{proof}[Proof of Theorem \ref{thm:knm_small_d}]

First, we deal with the $d=2$ case. Let $A = \{1,2,3,4,5\}$ and $B = \{6,7,8,9,10\}$ be the partite sets of the complete bipartite graph $K_{5,5}$. Let $q_1 = \lfloor\frac{n}{5}\rfloor$, and $q_2 = \lfloor\frac{m}{5}\rfloor$, and write $n = 5q_1 + r_1$, $m = 5q_2 + r_2$, where $0\leq r_1, r_2 < 5$. Define $a:A\cup B\to \mathbb{N}$ by
\begin{equation*}
a(v) = 
\begin{cases}
    q_1+1 & \text{if } 1\le v\le r_1,\\
    q_1 & \text{if } r_1+1\le v\le 5,\\
    q_2+1 & \text{if } 6\le v\le r_2+5,\\
    q_2 & \text{if } r_2+6\le v\le 10,
\end{cases}
\end{equation*}
for all $v\in A\cup B$. Note that $K_{5,5}^{(a)} = K_{n,m}$. Define $p: A\cup B \rightarrow \mathbb{R}^2$ by
\begin{equation*}
p(v) = 
\begin{cases}
    (0, 0) & \text{if $v=1$},\\
    (-21.1521153004, 46.3456374209) & \text{if $v=2$},\\
    (0.1564190972, 0.421467513) & \text{if $v=3$},\\ 
    (-67.6050650164, 24.0271141104) & \text{if $v=4$},\\
    (0.0496284145, 0.1312332404) & \text{if $v=5$},\\ 
    (-0.0000017871, 0.0000006617) & \text{if $v=6$},\\ 
    (-21.1521097826, 46.3456444883) & \text{if $v=7$},\\ 
    (0.1564134151, 0.4214695353) & \text{if $v=8$},\\ 
    (-67.6050626001, 24.0271500987) & \text{if $v=9$},\\
    (0.0496445695, 0.131227325) & \text{if $v=10$}.
\end{cases}
\end{equation*}
By computer calculation, it can be checked that
\[
\min \left\{\lambda_{4}(L(K_{5,5},p)), \, \min_{v\in A\cup B} \lambda_1(L^v(K_{5,5},p))\right\} \ge 1.39.
\]
Let
\[
    h(a) = \frac{ \min_{\{u,v\}\in E(K_{5,5})} a(u)\cdot a(v)}{\max_{w\in V(K_{5,5})} a(w)}\ge \frac{q_1\cdot q_2}{\max\{q_1+1,q_2+1\}}\ge \min\{q_1,q_2\} - 1,
\]
and
\[
    g(a)= \min_{v\in V(K_{5,5})} a(v) = \min\{q_1,q_2\}.
\]
Then, by Lemma \ref{lemma:eigenvalue_bound}, we obtain
\[
    a_2(K_{n,m})\ge 1.39\cdot (\min\{q_1, q_2\}-1) \ge 1.39\cdot \left(\min\left\{\frac{n}{5}, \frac{m}{5}\right\}-2\right) = 0.278\cdot \min\{n,m\} - 2.78.
\]

Next, we deal with the $d=3$ case. As before, let $A=\{1,2,3,4,5\}$ and $B=\{6,7,8,9,10\}$ be the partite sets of $K_{5,5}$. 
We define $p:A\cup B\rightarrow\mathbb{R}^3$ by
\begin{equation*}
p(v) = 
\begin{cases}
    (0, 0, 0) & \text{if $v=1$},\\
    (2.7293408266, -5.6302802394, 5.4822881749) & \text{if $v=2$},\\
    (2.06615403, -1.6992607629, 0.2852673824) & \text{if $v=3$},\\ 
    (8.7625036487, -3.3438898813, -1.0943816281) & \text{if $v=4$},\\
    (5.9302132437, 3.0124055593, 6.0022584873) & \text{if $v=5$},\\ 
    (2.8954859979, -5.2351324955, 5.5842950438) & \text{if $v=6$},\\ 
    (6.197393444, 2.9766904751, 6.27433625) & \text{if $v=7$},\\ 
    (6.96976192, -5.2762513597, 2.3713993733) & \text{if $v=8$},\\ 
    (0.7519782126, 0.1600313582, -0.8109033932) & \text{if $v=9$},\\
    (8.3720044593, -1.2879449195, 1.0858012249) & \text{if $v=10$}.
\end{cases}
\end{equation*}
By computer calculation, it can be checked that
\[
    \min \left\{\lambda_{7}(L(K_{5,5},p)), \, \min_{v\in A\cup B} \lambda_1(L^v(K_{5,5},p))\right\} \ge 0.309.
\]
Applying the same argument as before, we obtain
\[
    a_3(K_{n,m})\geq 0.0618 \cdot \min\{n,m\} - 0.618.
\]
\end{proof}

\begin{remark}
    The decision to use the graph $K_{5,5}$ in the proof of Theorem \ref{thm:knm_small_d} was, in a sense, arbitrary. 
    It is possible to find (with the help of a computer) good embeddings of larger complete bipartite graphs, which would provide, by repeating the arguments in the proof of Theorem \ref{thm:knm_small_d},  better lower bounds for $a_2(K_{n,m})$ and $a_3(K_{n,m})$ (for large enough $n$ and $m$).  However, since these improved bounds would still be suboptimal, and for the sake of simplicity, we chose to use the relatively small graph $K_{5,5}$ instead. Similarly, note that it is possible to prove, using the same method, explicit lower bounds for $a_d(K_{n,m})$ for any relatively small value of $d$. Again, for the sake of simplicity, we chose to focus on the cases $d=2$ and $d=3$.
\end{remark}

\subsection{The complete bipartite graph $K_{3,3}$}\label{sec:k33}

In this section we prove Theorem \ref{thm:k33_tight}, giving a conjecturally tight lower bound on the $2$-dimensional algebraic connectivity of $K_{3,3}$. The proof relies on the analysis of the limiting behavior of certain sequence of embeddings of $K_{3,3}$ in the plane, as described next.

Let $K_{3,3}=(V,E)$ be the complete graph with $3$ vertices on each side. For convenience, let $V = [6]$ be its vertex set and 
$
E = \{\{i,j\} :\, 1\le i,j\le 6,\, i \text{ odd},\, j \text{ even}\}
$ be its edge set. For brevity, we occasionally denote an edge $\{i,j\}\in E$, with $i<j$, by $ij$.

Let $0< \alpha,\beta< \pi/2$ and $c>0$. We define $p_{\alpha,\beta,c}: [6]\to \Rea^2$ as follows (see Figure \ref{fig:p_alpha_beta}).
\[
p_{\alpha,\beta,c}(1) = c\begin{pmatrix} \cos\alpha \\ 0 \end{pmatrix}, \quad
p_{\alpha,\beta,c}(3) = c\begin{pmatrix} 0 \\ \sin\alpha \end{pmatrix}, \quad
p_{\alpha,\beta,c}(5) = c\begin{pmatrix} 0 \\ -\sin\alpha \end{pmatrix},
\]
and
\[
p_{\alpha,\beta,c}(2) = p_{\alpha,\beta,c}(1)+ \begin{pmatrix} 1 \\ 0 \end{pmatrix}, \quad
p_{\alpha,\beta,c}(4) = p_{\alpha,\beta,c}(3)+\begin{pmatrix} \cos\beta \\ \sin\beta \end{pmatrix}, \quad
p_{\alpha,\beta,c}(6) = p_{\alpha,\beta,c}(5)+\begin{pmatrix} \cos\beta \\ -\sin\beta \end{pmatrix}.
\]

\begin{figure}
    \centering
    \begin{tikzpicture}

\coordinate (A) at (0.5,0);  % Vertex 1
\coordinate (B) at (-2.3,-1.5);  % Vertex 5
\coordinate (C) at (-2.3,1.5);  % Vertex 3

\coordinate (D) at (1.5,0);  % Vertex 2
\coordinate (E) at (-1.3,2.3);  % Vertex 4
\coordinate (F) at (-1.3,-2.3);  % Vertex 6
\coordinate (X) at (-2.3,0);  % Vertex X (invisible)
\coordinate (Y) at (-1.3,1.5);  % Vertex Y (invisible)

% Draw the vertices
\fill (A) circle (2pt);  % Vertex 1
\fill (B) circle (2pt);  % Vertex 5
\fill (C) circle (2pt);  % Vertex 3
\fill (D) circle (2pt);  % Vertex 2
\fill (E) circle (2pt);  % Vertex 4
\fill (F) circle (2pt);  % Vertex 6

% Label the vertices
\node at (A) [below,font=\small] {1};
\node at (B) [below left,font=\small] {5};
\node at (C) [above left,font=\small] {3};
\node at (D) [below right,font=\small] {2};
\node at (E) [above left,font=\small] {4};
\node at (F) [below left,font=\small] {6};

% Draw the edges
\draw[] (A) -- (D);  % Edge from 1 to 2
\draw[] (C) -- (E);  % Edge from 3 to 4
\draw[] (B) -- (F);  % Edge from 5 to 6
\draw[] (A) -- (E);  % Edge from 1 to 4
\draw[] (A) -- (F);  % Edge from 1 to 6
\draw[] (D) -- (C);  % Edge from 2 to 3
\draw[] (D) -- (B);  % Edge from 2 to 5
\draw[] (C) -- (F);  % Edge from 3 to 6
\draw[] (E) -- (B);  % Edge from 4 to 5

\draw[dashed] (X) -- (C) node[midway, left,font=\small] {};  % Edge from X to 3 with label "a"
\draw[dashed] (A) -- (C) node[midway, left,font=\small] {};  % Edge from X to 3 with label "a"

\draw[dashed] (X) -- (A) node[midway, left] {};  % Edge from X to 3 with label "a"
\draw[dashed] (Y) -- (C) node[midway, left] {};  % Edge from X to 3 with label "a"
\draw[dashed] (Y) -- (E) node[midway, right,font=\small] {};  % Edge from X to 3 with label "a"

\pic [draw, -, "$\alpha$", angle eccentricity=1.3, font=\small] {angle = C--A--X};
\pic [draw, -, "$\beta$", angle eccentricity=1.3, font=\small] {angle = Y--C--E};

\end{tikzpicture}
    \caption{The embedding $p_{\alpha,\beta,c}$}
    \label{fig:p_alpha_beta}
\end{figure}
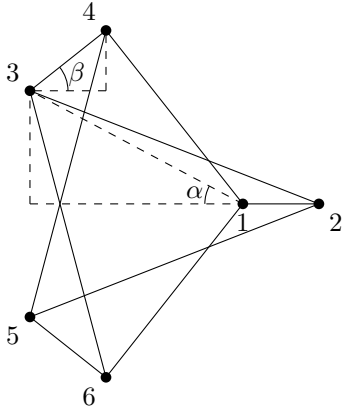

Fix $0< \alpha,\beta< \pi/2$. Let $\lminusva{}{\alpha,\beta}$ be the entry-wise limit of the matrices $\lminusva{}{}(K_{3,3},p_{\alpha,\beta,c})$ as $c$ tends to infinity. 
Let $e,e'\in E$ such that $|e\cap e'|=1$.
It is not hard to check that
\begin{equation}\label{eq:angles_limit}
    \lim_{c\to\infty} \theta_{p_{\alpha,\beta,c}}(e,e')=\begin{cases}
        \alpha & \text{if } \{e,e'\}\in \left\{ \{12,23\},\{12,25\} \right\},     \\
        \pi-\alpha & \text{if } \{e,e'\}\in\left\{ \{12,14\},\{12,16\} \right\}, \\ 
        2\alpha & \text{if } \{e,e'\}\in\left\{ \{14,16\},\{23,25\}   \right\}, \\ 
        \pi/2-\alpha  & \text{if } \{e,e'\}\in\left\{ \{14,45\},\{25,45\},\{16,36\},\{23,36\}   \right\}, \\ 
        \pi/2-\beta  & \text{if } \{e,e'\}\in\left\{ \{34,45\},\{36,56\}   \right\}, \\ 
        \pi/2+ \beta & \text{if } \{e,e'\}\in\left\{ \{34,36\},\{45,56\} \right\}, \\ 
        \alpha+\beta & \text{if } \{e,e'\}\in\left\{ \{23,34\},\{25,56\}  \right\}, \\ 
        \pi-\alpha-\beta  & \text{if } \{e,e'\}\in\left\{ \{14,34\},\{16,56\}  \right\}.
    \end{cases}
\end{equation}

Let $a=\sin(\alpha)$ and $b=\sin(\beta)$. Let $f(a,b)= \sqrt{(1-a^2)(1-b^2)}-ab$. 
By Lemma \ref{lemma:lower_weighted_stiffness}, \eqref{eq:angles_limit}, and standard trigonometric identities, we obtain 
\[
    \lminusva{}{\alpha,\beta}=\begin{pmatrix}
        2 & 0 & 0 & 0 & 0 & -\sqrt{1-a^2} &-\sqrt{1-a^2} &\sqrt{1-a^2}  &\sqrt{1-a^2}\\
        0 & 2 & 0 & -b & b & -f(a,b) & 0 & f(a,b) & 0\\
        0 & 0 & 2 & b & -b & 0 & -f(a,b) & 0 & f(a,b)\\
        0 & -b & b & 2 & 0 & 0 & a & a & 0\\
        0 & b &-b & 0 &2 & a & 0 & 0& a\\
        -\sqrt{1-a^2} & -f(a,b) & 0 & 0 &a & 2 & 1-2a^2 & 0 & 0\\
        -\sqrt{1-a^2} & 0 & -f(a,b) & a & 0 & 1-2a^2 &2 & 0 & 0\\
        \sqrt{1-a^2} & f(a,b) & 0 & a& 0 &0 & 0 &2 & 1-2a^2\\
        \sqrt{1-a^2} & 0 & f(a,b) & 0& a & 0 &0 & 1-2a^2 & 2
    \end{pmatrix},
\]
where the rows and columns are indexed by the edges of $K_{3,3}$, ordered as:  $12, 34, \allowbreak 56, 36,\allowbreak 45,\allowbreak 14, \allowbreak 16, 23, 25$.

\begin{lemma}\label{lemma:k33_ab_spectrum}
Let $0<\alpha,\beta<\pi/2$. Let $a=\sin(\alpha)$ and $b=\sin(\beta)$, and let $f(a,b)=\sqrt{(1-a^2)(1-b^2)}-ab$. Then, 
\[
    \Spec(\lminusva{}{\alpha,\beta})= \{2(1-a^2),1+2a^2,2,3,\theta_{1,a,b},\theta_{2,a,b},\eta_{1,a,b},\eta_{2,a,b},\eta_{3,a,b}\},
\]
where $\theta_{1,a,b},\theta_{2,a,b}$ are the roots of the polynomial
\[
   p_{1,a,b}(x)= x^2 + (2a^2-5)x -2f(a,b)^2+2,
\]
and $\eta_{1,a,b},\eta_{2,a,b},\eta_{3,a,b}$ are the roots of the polynomial
\[
p_{2,a,b}(x)= x^3 - (2a^2+5)x^2 +2(3a^2-2b^2-f(a,b)^2+4)x + 8a^2 (b^2-1).
\]
\end{lemma}
\begin{proof}
Let
\[
    P=\begin{pmatrix}
        0 & 0 & 0 & f(a,b) & 2\sqrt{1-a^2} & 0 & 0 & 0 & 0 \\
        0 & 0 & 0 & -\sqrt{1-a^2} & f(a,b) & 0 & 1 & 0 & 0 \\
        0 & 0 & 0 & -\sqrt{1-a^2} & f(a,b) & 0 & -1 & 0 & 0 \\
        1 & 2a & 0 & 0 &0 &0 & 0 & 1 & 0\\
        1 & 2a & 0 & 0 &0 &0 &0 &-1 & 0 \\
        -a & 1 & 1  & 0 &0 & 1 & 0 & 0 &1\\
        -a & 1 & -1 & 0 &0 &1 & 0 & 0 &-1 \\
        -a & 1 & 1 & 0 & 0& -1 & 0 & 0 &-1\\
        -a & 1 & -1 & 0 & 0 &-1 &0 &0 & 1
    \end{pmatrix}.
\]
It is easy to verify that $P$ is non-singular. 
Let $B$ be the basis of $\Rea^{|E|}$ consisting of the columns of $P$ (where, as before, we order the edges of $K_{3,3}$ as $12, 34, \allowbreak 56, 36,\allowbreak 45,\allowbreak 14, \allowbreak 16, 23, 25$). Then, it is not hard to check that the matrix representation of $\lminusva{}{\alpha,\beta}$ with respect to $B$ is
\[
    P^{-1} \lminusva{}{\alpha,\beta} P = \begin{pmatrix}
        2(1-a^2) & 0 &0 &0 &0 &0 &0 &0 &0 \\
        0 & 3 & 0 &0 &0 &0 &0 &0 &0\\
        0 & 0 & 1+2a^2 & 0 &0 &0 &0 &0 &0\\
        0 &0 &0 & 2 & 0 &0 &0 &0 &0\\
        0 &0 &0 &0 & 2 & -2 & 0 &0 &0\\
        0 &0 &0 &0 & 2a^2-2-f(a,b)^2 & 3-2a^2 & 0 &0 &0\\
        0 & 0 &0 &0 &0 &0 & 2 & -2b & -2f(a,b)\\
        0 & 0 &0 &0 &0 &0 &-2b & 2 &-2a\\
        0 & 0 &0 &0 &0 &0 & -f(a,b) & -a & 2a^2+1
    \end{pmatrix}.
\]
Hence, the spectrum of $\lminusva{}{\alpha,\beta}$ consists of the eigenvalues $2(1-a^2), 3, 1+2a^2, 2$, together with the spectra of the two matrices
\[
    M_{1,a,b} =\begin{pmatrix}
         2 & -2 \\ 2a^2-2-f(a,b)^2 & 3-2a^2
    \end{pmatrix}
\]
and
\[
    M_{2,a,b}= \begin{pmatrix}
        2 & -2b & -2f(a,b)\\
        -2b & 2 &-2a\\
        -f(a,b) & -a & 2a^2+1
    \end{pmatrix}.
\]
Finally, the claim follows by noticing that the characteristic polynomial of $M_{1,a,b}$ is
\[
   p_{1,a,b}(x)= x^2 + (2a^2-5)x -2f(a,b)^2+2,
\]
and the characteristic polynomial of $M_{2,a,b}$ is
\[
p_{2,a,b}(x)= x^3 - (2a^2+5)x^2 +2(3a^2-2b^2-f(a,b)^2+4)x + 8a^2 (b^2-1).
\]
\end{proof}

In order to obtain the best possible lower bound on $a_2(K_{3,3})$, we need to identify the values of $\alpha$ and $\beta$ (or equivalently,  $a$ and $b$) for which the smallest eigenvalue of $\lminusva{}{\alpha,\beta}$ is maximized.
Numerical experiments suggest that the smallest eigenvalue of $\lminusva{}{\alpha,\beta}$ is maximized when $a\approx 0.830893$ and $b\approx 0.314632$, and moreover that in this case $2(1-a^2)$ is a root of both characteristic polynomials $p_{1,a,b}(x)$ and $p_{2,a,b}(x)$. Thus motivated, we proceed to find the exact values of $0<a,b<1$ for which $p_{1,a,b}(2(1-a^2))=p_{2,a,b}(2(1-a^2))=0$ and analyze the eigenvalues of $\lminusva{}{\alpha,\beta}$ in this case.

\begin{lemma}\label{lemma:k33_optimal_a_b}
   Let $\lambda\approx 0.6903845$ be the unique positive real root of the polynomial $176 x^4-200 x^3+47 x^2+18 x-9$. Let $a_0=\sqrt{\lambda}\approx 0.830893$ and $b_0=\sqrt{ 6 a_0^4 - 8a_0^2 +3 +2a_0(a_0^2 - 1)\sqrt{9a_0^2 - 6}}\approx 0.314632$. Let $\alpha_0=\sin^{-1}(a_0)$ and $\beta_0=\sin^{-1}(b_0)$. Then, 
   \[
        \lambda_1(\lminusva{}{\alpha_0,\beta_0}) = 2(1-a_0^2)=2(1-\lambda)\approx 0.6192309.
   \]
\end{lemma}
\begin{proof}

By Lemma \ref{lemma:k33_ab_spectrum}, $\lambda_1(\lminusva{}{\alpha_0,\beta_0})$ is either $2(1-a_0^2)$, the smallest root of $p_{1,a_0,b_0}(x)$, or the smallest root of $p_{2,a_0,b_0}(x)$,
where
\[
   p_{1,a_0,b_0}(x)= x^2 + (2a_0^2-5)x -2f(a_0,b_0)^2+2,
\]
and 
\[
p_{2,a_0,b_0}(x)= x^3 - (2a_0^2+5)x^2 +2(3a_0^2-2b_0^2-f(a_0,b_0)^2+4)x + 8a_0^2 (b_0^2-1).
\]
We will show that $\lambda_1(\lminusva{}{\alpha_0,\beta_0})=2(1-a_0^2)$. In fact, we will show that $2(1-a_0^2)$ is the smallest root of both $p_{1,a_0,b_0}(x)$ and $p_{2,a_0,b_0}(x)$.

First, we show that $2(1-a_0^2)$ is a root of $p_{1,a_0,b_0}(x)$. Note that, by Vieta's formulas, the sum of the two roots of $p_{1,a_0,b_0}(x)$ is $5-2a_0^2$. Hence, $2(1-a_0^2)$ is a root of $p_{1,a_0,b_0}(x)$ if and only if $3$ is a root. Solving $p_{1,a,b}(3)=0$, we obtain
\[
f(a,b)^2= 3a^2-2.
\]
Solving for $b$, we obtain the formula
\[
    b= \pm \sqrt{ 6 a^4 - 8a^2 +3 \pm 2a(a^2 - 1)\sqrt{9a^2 - 6} }.
\]
In particular, by the definition of $b_0$, we have  $f(a_0,b_0)=3a_0^2-2$, and therefore $p_{1,a_0,b_0}(3)=0$, as wanted.
Note that, since $2(1-a_0^2)<3$, the smallest root of $p_{1,a_0,b_0}(x)$ is $2(1-a_0^2)$.

Next, we want to show that $p_{2,a_0,b_0}(2(1-a_0^2))=0$. 
That is, we need to show that
\[
    (2(1-a_0^2))^3 -(2a_0^2+5)(2(1-a_0^2))^2 +2(3a_0^2-2b_0^2-f(a_0,b_0)^2+4)\cdot 2(1-a_0^2) + 8 a_0^2 (b_0^2-1)=0.
\]
Plugging in $f(a_0,b_0)=3a_0^2-2$ and $b_0^2=6 a_0^4 - 8a_0^2 +3 + 2a_0(a_0^2 - 1)\sqrt{9a_0^2 - 6}$, we obtain, after some simplification, the equation
\[
80 a_0^6 -156 a_0^4 +88 a_0^2 -12 +(32a_0^5-48 a_0^3+16 a_0)\sqrt{9a_0^2-6}=0.
\]
Noting that the polynomials $80a^6-156a^4+88a^2-12$ and $32a^5-48 a^3+16 a$ are both divisible by $a^2-1$, we obtain
\[
80a_0^4-76a_0^2+12 +(32a_0^3-16 a_0)\sqrt{9a_0^2-6}=0.
\]
Since $80a^4-76a^2+12\le 0$ and $32a^3-16a \ge 0$ for $1/\sqrt{2}\le a\le \sqrt{3}/2$ (and in particular for $a=a_0$), this is equivalent to
\[
(80 a_0^4 -76 a_0^2 +12)^2 =(32a_0^3-16 a_0)^2 (9a_0^2-6).
\]
Simplifying again, we obtain
\[
    -16 (176 a_0^8 - 200 a_0^6 + 47 a_0^4 + 18 a_0^2 - 9)=0.
\]
By the definition of $a_0$, this equation is indeed satisfied. Hence, $2(1-a_0^2)$ is a root of $p_{2,a_0,b_0}(x)$.

We are left to show that the other two roots of $p_{2,a_0,b_0}(x)$ are greater or equal than $2(1-a_0^2)$. We denote these two roots by $x_1$ and $x_2$. By Vieta's formulas, we have
\[
    \begin{cases}
        2(1-a_0^2)+x_1+x_2 = 2a_0^2+5,\\
        2(1-a_0^2) x_1 x_2 = 8a_0^2(1-b_0^2).
    \end{cases}
\]
Solving this system of equations, we obtain
\[
    x_1 = 2a_0^2+\frac{3}{2} -\frac{1}{2} \sqrt{-80 a_0^4 - 32 a_0^3 \sqrt{9 a_0^2 - 6} + 56 a_0^2 + 9}
\]
and
\[
    x_2 = 2a_0^2+\frac{3}{2} +\frac{1}{2} \sqrt{-80 a_0^4 - 32 a_0^3 \sqrt{9 a_0^2 - 6} + 56 a_0^2 + 9}.
\]
It is not hard to verify that, for $a_0\approx 0.830893$, $x_1,x_2> 2(1-a_0^2)$. Therefore, \[\lambda_1(\lminusva{}{\alpha_0,\beta_0})=2(1-a_0^2)=2(1-\lambda).\]
\end{proof}

\begin{proof}[Proof of Theorem \ref{thm:k33_tight}]
     Let $\lambda\approx 0.6903845$ be the unique positive real root of the polynomial $176 x^4-200 x^3+47 x^2+18 x-9$. Let $a_0=\sqrt{\lambda}\approx 0.830893$ and $b_0=\sqrt{ 6 a_0^4 - 8a_0^2 +3 +2a_0(a_0^2 - 1)\sqrt{9a_0^2 - 6}}\approx 0.314632$. Let $\alpha_0=\sin^{-1}(a_0)$ and $\beta_0=\sin^{-1}(b_0)$. By Lemma \ref{lemma:lower_stiffness_spectrum}, the continuity of eigenvalues, and Lemma \ref{lemma:k33_optimal_a_b}, we obtain
     \begin{align*}
     a_2(K_{3,3}) &\ge \sup_{c}\lambda_4(L(K_{3,3},p_{\alpha_0,\beta_0,c})) \ge \lim_{c\to\infty}\lambda_4(L(K_{3,3},p_{\alpha_0,\beta_0,c}))   \\ &= \lim_{c\to\infty} \lambda_1(\lminusva{}{}(K_{3,3},p_{\alpha_0,\beta_0,c}))= \lambda_1(\lminusva{}{\alpha_0,\beta_0})=2(1-\lambda).
     \end{align*}
\end{proof}

\subsubsection{A different embedding of $K_{3,3}$ in the plane}

Next, we examine an alternative embedding of $K_{3,3}$ in the plane. Although this embedding is not optimal in terms of its spectral gap,  the resulting stiffness matrix exhibits a curious half-integral spectrum, which may be of interest.

As before, denote the vertex set of $K_{3,3}$ by $V=[6]$ and its edge set by $E=\{\{i,j\}:\, 1\le i,j\le 6, \, i \text{ odd}, \, j \text{ even}\}$. 
Let $x_1$, $x_2$, $x_3$ be the vertices of an equilateral triangle in $\mathbb{R}^2$. Let $y_1$ be the midpoint between $x_1$ and $x_2$, $y_2$ be the midpoint between $x_2$ and $x_3$, and $y_3$ be the midpoint between $x_1$ and $x_3$. 
We define $\hat{p}: [6] \rightarrow \mathbb{R}^2$ (see Figure \ref{fig:1}) by
\begin{equation*}
\hat{p}(i) = 
\begin{cases}
    x_{(i+1)/2} & \text{if $i$ is odd,}\\
    y_{i/2} & \text{if $i$ is even.}\\
\end{cases}
\end{equation*}

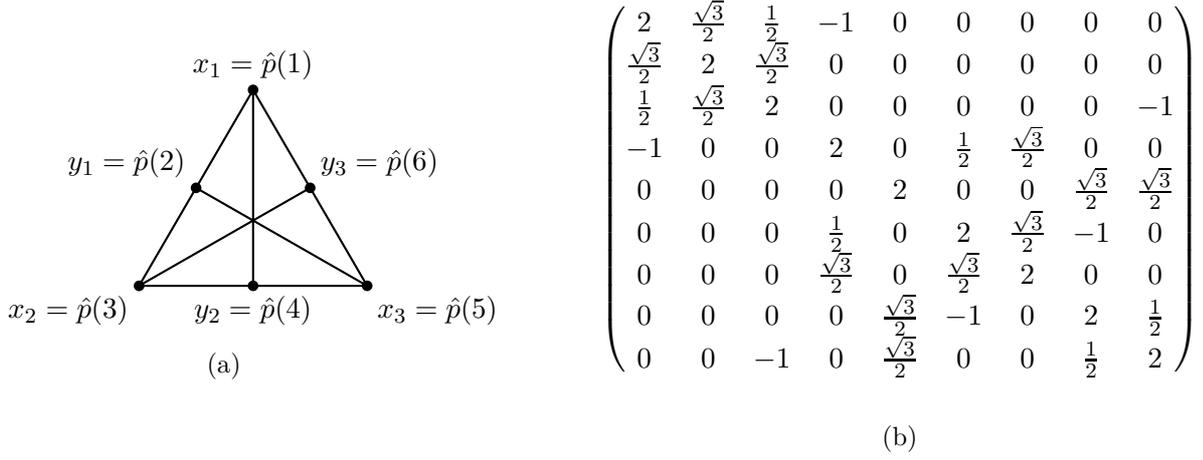
\begin{figure}[ht]
    \centering
    \begin{minipage}{0.35\textwidth}
        \centering
       \begin{tikzpicture}
    % Coordinates of the vertices (equilateral triangle)
    \coordinate (x1) at (0, {sqrt(27/4)});  % Top vertex
    \coordinate (x2) at (-1.5, 0);  % Bottom-left vertex
    \coordinate (x3) at (1.5, 0);   % Bottom-right vertex

    % Midpoints
    \coordinate (y1) at ($(x1)!0.5!(x2)$); % Midpoint of x1 and x2
    \coordinate (y2) at ($(x2)!0.5!(x3)$); % Midpoint of x2 and x3
    \coordinate (y3) at ($(x1)!0.5!(x3)$); % Midpoint of x1 and x3

    % Draw the equilateral triangle
    \draw[thick] (x1) -- (x2) -- (x3) -- cycle;
    
    % Draw the additional edges
    \draw[thick] (y1) -- (x3);
    \draw[thick] (y2) -- (x1);
    \draw[thick] (y3) -- (x2);

    % Mark the vertices and midpoints with labels
    \fill (x1) circle (2pt);
    \node at (x1) [above] {$x_1=\hat{p}(1)$};
    
    \fill (x2) circle (2pt);
    \node at (x2) [below left] {$x_2=\hat{p}(3)$};
    
    \fill (x3) circle (2pt);
    \node at (x3) [below right] {$x_3=\hat{p}(5)$};
    
    \fill (y1) circle (2pt);
    \node at (y1) [above left] {$y_1=\hat{p}(2)$};
    
    \fill (y2) circle (2pt);
    \node at (y2) [below] {$y_2=\hat{p}(4)$};
    
    \fill (y3) circle (2pt);
    \node at (y3) [above right] {$y_3=\hat{p}(6)$};
    
\end{tikzpicture}\subcaption{}
    \end{minipage}  
    \hfill
    \begin{minipage}{0.6\textwidth}
        \centering
        \[
        \begin{pmatrix}
            2& \frac{\sqrt{3}}{2} & \frac{1}{2}  & -1  & 0 & 0 &0  & 0 & 0 \\
            \frac{\sqrt{3}}{2}& 2 & \frac{\sqrt{3}}{2}  &0  &0  & 0 &  0& 0 &0  \\
            \frac{1}{2} & \frac{\sqrt{3}}{2}  & 2 &0  & 0 & 0 & 0 &  0&-1  \\
           -1 & 0 & 0  & 2 & 0  & \frac{1}{2} & \frac{\sqrt{3}}{2} & 0  & 0 \\
            0& 0 & 0 &0  &2  & 0 & 0 &\frac{\sqrt{3}}{2}  & \frac{\sqrt{3}}{2} \\
           0 &0  &0  & \frac{1}{2} & 0 &2  & \frac{\sqrt{3}}{2} & -1  &0  \\
          0  &0  &0  & \frac{\sqrt{3}}{2}& 0 & \frac{\sqrt{3}}{2} &2  & 0 & 0 \\
          0  &0  & 0 & 0 & \frac{\sqrt{3}}{2} &  -1 &  0 &2  &\frac{1}{2}  \\
           0 &0  & -1 & 0 & \frac{\sqrt{3}}{2} & 0 & 0 &\frac{1}{2}  &2  \\
        \end{pmatrix}
        \]
        \subcaption{}
    \end{minipage}
    \caption{(a) The embedding $\hat{p}$. (b) The lower stiffness matrix $L^{\downarrow}(K_{3,3},\hat{p})$ (where the rows and columns are indexed by the edges in $E$, ordered lexicographically).}\label{fig:1}
\end{figure}

Let $M\in \Rea^{n \times n}$ be a symmetric matrix, and let $\lambda_1, \ldots, \lambda_k$ be its eigenvalues, with multiplicities $m_1, \ldots, m_k$ respectively. Then, we denote the spectrum of $M$ as 
\[
    \text{Spec}(M) = \{\lambda_1^{[m_1]}, \ldots, \lambda_k^{[m_k]}\}.
\]

\begin{proposition}\label{prop:k3k3spectrum}
    The spectrum of $L(K_{3,3}, \hat{p})$ is 
    \[
    \left\{0^{[3]},0.5^{[3]}, 1.5^{[2]}, 2.5^{[1]}, 3^{[1]}, 4^{[2]} \right\}.
    \]
\end{proposition}
\begin{proof}
Denote $\lminusva{}{} = \lminusva{}{}(K_{3,3}, \hat{p})$. By Lemma \ref{lemma:lower_stiffness_spectrum}, the claim is equivalent to showing that the spectrum of $\lminusva{}{}$ is 

\[
\left\{0.5^{[3]}, 1.5^{[2]}, 2.5^{[1]}, 3^{[1]}, 4^{[2]} \right\}.
\]

Let $e,e'\in E$ with $|e\cap e'|=1$. It is easy to check that 
\begin{equation}\label{eq:k33_angles}
    \cos(\theta_{\hat{p}}(e,e'))=\begin{cases}
        0 & \text{if } \{e,e'\}\in \left\{ \{12,25\},\{23,25\},\{34,14\},\{45,14\},\{16,36\},\{56,36\}\right\},\\
        \frac{1}{2} & \text{if } \{e,e'\}\in\left\{ \{12,16\},\{45,56\},\{23,34\}\right\},\\
        -1 & \text{if } \{e,e'\}\in \left\{ \{12,23\},\{34,45\},\{56,16\}\right\},\\
        \frac{\sqrt{3}}{2} & \text{otherwise.}
    \end{cases}
\end{equation}

We proceed to find a basis of $\Rea^{|E|}$ consisting of eigenvectors of $\lminusva{}{}$. 

\textbf{Eigenvalue 2.5}: Let $\psi_{2.5} \in \mathbb{R}^{|E|}$ be defined by 
\begin{equation*}
(\psi_{2.5})_{e}=
\begin{cases}
    1 & \text{if } e\in \{12, 34, 56\},\\
    -1 & \text{if } e\in \{23, 45, 16\},\\
       0 & \text{otherwise,} 
\end{cases}
\end{equation*}
for all $e\in E$. It is easy to check using Lemma \ref{lemma:lowerlaplacian} and \eqref{eq:k33_angles} that $\psi_{2.5}$ is an eigenvector of $\lminusva{}{}$ with eigenvalue $2.5$.

\textbf{Eigenvalue 3}: Let $\psi_{3} \in \mathbb{R}^{|E|}$ be defined by 
\begin{equation*}
(\psi_{3})_{e}=
\begin{cases}
    \sqrt{3} & \text{if $e \in \{14, 25, 36\}$},\\
    1 & \text{otherwise}.\\
\end{cases}
\end{equation*}
It is easy to check using Lemma \ref{lemma:lowerlaplacian} and \eqref{eq:k33_angles} that $\psi_{3}$ is an eigenvector of $\lminusva{}{}$ with eigenvalue 3.

\textbf{Eigenvalue 1.5}: Let $\psi_{1.5} \in \mathbb{R}^{|E|}$ be defined by 
\begin{equation*}
(\psi_{1.5})_{e}=
\begin{cases}
    \sqrt{3} & \text{if $e = 36$},\\
    -\sqrt{3} & \text{if $e = 14$},\\
    1 & \text{if $e \in \{12,45\}$},\\
    -1 & \text{if $e \in \{23,56\}$},\\
    0 & \text{otherwise.}\\
\end{cases}
\end{equation*}

For $x\in\mathbb{N}$, let $h(x)$ be the unique number in $[6]$ that is congruent to $x$ modulo $6$. 
Let $\psi_{1.5}' \in \mathbb{R}^{|E|}$ be defined by
\[
(\psi_{1.5}')_{\{i,j\}} = (\psi_{1.5})_{\{h(i+2), h(j+2)\}},
\]
for all $\{i,j\}\in E$. 

It is easy to check using Lemma \ref{lemma:lowerlaplacian} and \eqref{eq:k33_angles} that $\psi_{1.5}$ and $\psi_{1.5}'$ are eigenvectors of $\lminusva{}{}$ with eigenvalue 1.5. Since $\psi_{1.5}$ and $\psi_{1.5}'$ are not scalar multiples of each other, the two vectors are linearly independent. 

\textbf{Eigenvalue 4}: Let $\psi_{4} \in \mathbb{R}^{|E|}$ be defined by 
\begin{equation*}
(\psi_{4})_{e}=
\begin{cases}
    -2\sqrt{3} & \text{if $e = 36$},\\
    \sqrt{3} & \text{if $e \in \{14,25\}$},\\
    3 & \text{if $e \in \{12, 45\}$},\\
    -4 & \text{if $e \in \{23, 34\}$},\\
    1 & \text{otherwise.}\\
\end{cases}
\end{equation*}

Let $\psi_{4}' \in \mathbb{R}^{|E|}$ be defined by
\[
(\psi_{4}')_{\{i,j\}} = (\psi_{4})_{\{h(i+2), h(j+2)\}},
\]
for all $\{i,j\}\in E$. It is easy to check using Lemma \ref{lemma:lowerlaplacian} and \eqref{eq:k33_angles} that $\psi_{4}$ and $\psi_{4}'$ are eigenvectors of $\lminusva{}{}$ with eigenvalue 4. Since $\psi_{4}$ and $\psi_{4}'$ are not multiples of each other, the two vectors are linearly independent. 

\textbf{Eigenvalue 0.5}: Let $\psi_{0.5} \in \mathbb{R}^{|E|}$ be defined by 
\begin{equation*}
(\psi_{0.5})_{e}=
\begin{cases}
    -2 & \text{if $e = 36$},\\
    -1 & \text{if $e \in \{14,25\}$},\\
    0 & \text{if $e \in \{16,56\}$},\\
    \sqrt{3} & \text{otherwise.}\\
\end{cases}
\end{equation*}

Let $\psi_{0.5}', \psi_{0.5}'' \in \mathbb{R}^{|E|}$ be defined by
\[
(\psi_{0.5}')_{\{i,j\}} = (\psi_{0.5})_{\{h(i+2), h(j+2)\}}
\]
and
\[
(\psi_{0.5}'')_{\{i,j\}} = (\psi_{0.5})_{\{h(i+4), h(j+4)\}},
\]
for all $\{i,j\}\in E$. It is easy to check using Lemma \ref{lemma:lowerlaplacian} and \eqref{eq:k33_angles} that $\psi_{0.5}$, $\psi_{0.5}'$ and $\psi_{0.5}''$ are eigenvectors of $\lminusva{}{}$ with eigenvalue 0.5. We are left to show that these vectors are linearly independent. 

Assume for the sake of contradiction that they are linearly dependent. Then, there exist scalars $\alpha, \beta, \gamma \in \mathbb R$, not all of them zero, such that
\[
\alpha\psi_{0.5} + \beta\psi_{0.5}' + \gamma\psi_{0.5}'' = 0.
\]
Plugging in $e = 56$, we obtain
\[
0=\alpha(\psi_{0.5})_{56} + \beta(\psi_{0.5}')_{56}+ \gamma(\psi_{0.5}'')_{56} = \sqrt{3}\beta + \sqrt{3}\gamma.
\]
Similarly, plugging in  $e=34$ and $e=12$, we obtain the equations
\[
    \sqrt{3}\alpha+\sqrt{3}\gamma=0
\]
and
\[
    \sqrt{3}\alpha+\sqrt{3}\beta=0,
\]
respectively. Solving this system of equations we obtain $\alpha = \beta = \gamma = 0$, a contradiction.

\end{proof}

\section{Generalized star graphs}\label{sec:star}

A graph $G=(V,E)$ is called \emph{minimally $d$-rigid} if it is $d$-rigid but, for all $e\in E$, the graph obtained from $G$ by removing the edge $e$ is not $d$-rigid. The following upper bound on the $d$-dimensional algebraic connectivity of minimally $d$-rigid graphs was proved in \cite{LNPR23+}.

\begin{proposition}[Lew, Nevo, Peled, Raz {\cite[Proposition 7.1]{LNPR23+}}]\label{prop:minimally_rigid}
  Let $d\ge 1$ and $n\ge d+1$, and let $T$ be a minimally $d$-rigid graph on $n$ vertices. Then, unless $d=1$ and $n=2$, or $d=2$ and $n=3$, 
\[
    a_d(T)\leq 1.
\]
\end{proposition}

Let $d\ge 1$ and $n\ge d+1$. Let $S_{n,d}$ be the graph on vertex set $[n]$ with edge set 
\[
E(S_{n,d}) = \left\{\{i,j\} : i\in [d], \, j\in [n]\setminus\{i\}\right\}.
\]
Note that $S_{n,1}$ is the star graph on $n$ vertices.  
It is easy to check that $S_{n,d}$ is minimally $d$-rigid. Let $e_1, e_2, \ldots, e_d$ be the standard basis vectors in $\Rea^d$. Let $p^* : [n] \rightarrow\mathbb R^d$ be defined by
\begin{equation*}
    p^*(i) = \begin{cases}
        e_i & \text{if $1\leq i \leq d$},\\
        0 & \text{otherwise.}
    \end{cases}
\end{equation*}

The spectrum of the stiffness matrix $L(S_{n,d},p^{*})$ was determined in \cite{LNPR23+}.

\begin{proposition}[Lew, Nevo, Peled, Raz {\cite[Proposition 7.4]{LNPR23+}}]\label{prop:star_graph}
    Let $d\ge 1$ and $n\ge d+1$. Then, the spectrum of $L(S_{n,d},p^*)$ is
\[
\left\{ 0^{\left[\binom{d+1}{2}\right]}, 1^{\left[dn-\binom{d+1}{2}-d\right]}, (n-d/2)^{[d-1]}, n^{[1]} \right\}.
\]
\end{proposition}
From Propositions \ref{prop:minimally_rigid}
and \ref{prop:star_graph}, we obtain $a_d(S_{n,d})=1$ (except in the case when $d=1$ and $n=2$, where we have $S_{2,1}=K_2$ and $a_1(K_2)=2$, and the case when $d=2$ and $n=3$, where we have $S_{3,2}=K_3$ and $a_2(K_3)=3/2$; 
 see \cite[Proposition 7.6]{LNPR23+}).

Here, relying on Theorem \ref{thm:lower_stiffness_blow_up}, we present a new, simpler proof of Proposition \ref{prop:star_graph}.

\begin{proof}[Proof of Proposition \ref{prop:star_graph}]
    We work with the lower stiffness matrix $\lminusva{}{}(S_{n,d},p^*)$. By Lemma \ref{lemma:lower_stiffness_spectrum} (using the fact that $|E(S_{n,d})|=d n -\binom{d+1}{2}$), 
    the claim is equivalent to showing that the spectrum of $\lminusva{}{}(S_{n,d},p^*)$ is 
    \[
    \left\{ 1^{\left[dn-\binom{d+1}{2}-d\right]}, (n-d/2)^{[d-1]}, n^{[1]} \right\}.
    \]    
    First, notice that $(S_{n,d}, p^*)$ is a blow-up of the standard $d$-simplex. More precisely, we have
    \[
    (S_{n,d},p^*)= (K_{d+1}^{(a)},\blowup{p}{a}),
    \]
    where $p:[d+1]\to\Rea^d$ is defined by
    \[
        p(i)=\begin{cases}
            e_i & \text{if } 1\leq i \leq d,\\
            0 & \text{if } i=d+1,
        \end{cases}
    \]
    and $a:[d+1]\to \mathbb{N}$ is defined by
    \begin{equation*}
        a(i) = \begin{cases}
              1 & \text{if } 1\leq i\leq d,\\
            n-d & \text{if $i = d+1$}.
        \end{cases}
    \end{equation*}
    Moreover, note that, since $a(i)=1$ unless $i=d+1$, we have
    \[
        \sum_{\{i,j\}\in E(K_{d+1})} (a(i)-1)(a(j)-1)=0.
    \]
    Therefore, by Theorem \ref{thm:lower_stiffness_blow_up}, 
     \[
        \text{Spec}\left(\lminusva{}{}(S_{n,d},p^*)\right)= \left(\bigcup_{i\in [d+1]} \text{Spec}(\lminusva{i}{a}(K_{d+1},p))^{[a(i)-1]}\right)\cup \text{Spec}(\lminusva{}{a}(K_{d+1},p)).
    \]
    Since $a(i) = 1$ for $1\le i\le d$, we have
    \[
        \text{Spec}\left(\lminusva{}{}(S_{n,d},p^*)\right)=  \text{Spec}(\lminusva{d+1}{a}(K_{d+1},p))^{[n-d-1]}\cup \text{Spec}(\lminusva{}{a}(K_{d+1},p)).
    \]
    Using Lemma \ref{lemma:lower_local_stiffness}, it is easy to check that $\lminusva{d+1}{a}(K_{d+1},p) = I_d$, so 
    \[\text{Spec}(\lminusva{d+1}{a}(K_{d+1},p))^{[n-d-1]} = \left\{ 1\right\}^{\left[dn-d^2-d\right]}.\]
    Hence, we are left to show that
    \[
      \text{Spec}(\lminusva{}{a}(K_{d+1},p))=  \left\{ 1^{\left[\frac{d(d-1)}{2}\right]}, (n-d/2)^{[d-1]}, n^{[1]} \right\}.
    \]
    Let $e,e'\in E(K_{d+1})$ such that $|e\cap e'|=1$. Note that
    \begin{equation}\label{eq:theta1}
        \cos(\theta_p(e,e')) = \begin{cases}
            \frac{1}{2} & \text{if } d+1\notin e\cup e',\\
            \frac{\sqrt{2}}{2} & \text{if } d+1\in e\cup e',\, d+1\notin e\cap e',\\
            0 & \text{otherwise.}
        \end{cases}
    \end{equation}

    For convenience, let $\lminusva{}{a}=\lminusva{}{a}(K_{d+1},p)$, and let $k=n-d$. Note that $k>0$. We may identify the set of edges of $K_{d+1}$ that contain $d+1$ with the set $[d]$ (by mapping each such edge $e$ to the unique vertex in $e\setminus\{d+1\}$). Then, ordering the edges of $K_{d+1}$ so that  the edges containing $d+1$ appear first, we obtain, by Lemma \ref{lemma:lower_weighted_stiffness} and \eqref{eq:theta1},
    \[
    \lminusva{}{a} = 
    \begin{pmatrix}
    (k+1)I_d & \sqrt{\frac{k}{2}}A \\
    \sqrt{\frac{k}{2}}A^{\T} & L' 
    \end{pmatrix},
    \]
    where $L' \in \mathbb R^{{d\choose 2}\times {d\choose 2}}$ is the matrix defined by
    \begin{equation*}
        L'_{e,e'} = \begin{cases}
            2 & \text{if }e = e',\\
            \frac{1}{2} & \text{if } |e\cap e'|=1,\\
            0 & \text{otherwise,}
        \end{cases}
    \end{equation*}
    for all $e,e'\in\binom{[d]}{2}$, and $A \in \mathbb R^{d\times {d\choose 2}}$ is defined by
    \begin{equation*}
        A_{v,e} = \begin{cases}
            1 & \text{if }v\in e,\\
            0 & \text{otherwise,}
        \end{cases}
    \end{equation*}
    for all $v\in [d]$ and $e\in\binom{[d]}{2}$.
   Note that $L' = \frac{1}{2}A^{\T}A + I_{{d\choose 2}}$.
   
    Let $\mathbf{v}\in \Rea^{d+1\choose 2}$ be an eigenvector of $\lminusva{}{a}$ with eigenvalue $\lambda$. Let $\mathbf{x}\in \Rea^d$ be the vector consisting of the first $d$ coordinates of $\mathbf{v}$, and let $\mathbf{y}\in \Rea^{d\choose 2}$ be the vector consisting of its last $\binom{d}{2}$ coordinates. Since $\lminusva{}{a}\mathbf{v} = \lambda\mathbf{v}$, we obtain the following system of equations.
     \begin{equation}\label{eq:star1}
        \begin{cases}
            (k+1)\mathbf{x} + \sqrt{\frac{k}{2}}A\mathbf{y} = \lambda\mathbf{x}\\
            \sqrt{\frac{k}{2}}A^{\T}\mathbf{x} + L'\mathbf{y} = \lambda\mathbf{y}.
        \end{cases}
    \end{equation}
    Rearranging the first equation in \eqref{eq:star1}, we obtain
    \begin{equation}\label{eq:star2}
         A\mathbf{y} = \sqrt{\frac{2}{k}}\cdot(\lambda-k-1)\mathbf{x}.
    \end{equation}
      Using $L' = \frac{1}{2}A^{\T}A + I_{{d\choose 2}}$, and then
    substituting \eqref{eq:star2} into the second equation in \eqref{eq:star1}, we obtain, after some simplification, the following system of equations.
\begin{equation}\label{eq:star3}
  \begin{cases}
           A\mathbf{y} = \sqrt{\frac{2}{k}}\cdot(\lambda-k-1)\mathbf{x}.\\
            (\lambda-1) A^{\T} \textbf{x}= \sqrt{2k}\cdot (\lambda-1) y.
        \end{cases}
\end{equation}
We divide into two cases. If $\lambda=1$, the second equation holds trivially, and we are left with a single equation $A\textbf{y}= -\sqrt{2k} \textbf{x}$. Therefore, for every $\textbf{y}\in \Rea^{\binom{d}{2}}$, $\textbf{v}=(-A\textbf{y}/\sqrt{2k},\textbf{y})$ is an eigenvector of $\lminusva{}{a}$ with eigenvalue $1$. That is, $\lambda=1$ is an eigenvalue of $\lminusva{}{a}$ with multiplicity $\binom{d}{2}=\frac{d(d-1)}{2}$, as wanted.

Now, assume $\lambda\ne 1$. Substituting the second equation in \eqref{eq:star3} into the first one, and simplifying, we obtain
\[
    \begin{cases}
        A A^{\T} \textbf{x} = 2 (\lambda-k-1) \textbf{x},\\
        A^{\T} \textbf{x} = \sqrt{2k} \textbf{y}.
    \end{cases}
\]
Therefore, $\textbf{v}$ is an eigenvector of $\lminusva{}{a}$ with eigenvalue $\lambda$ if and only if $\textbf{x}$ is an eigenvector of $A A^{\T}$ with eigenvalue $2(\lambda-k-1)$, and $\textbf{y}=A^{\T}\textbf{x}/\sqrt{2k}$. In particular, the multiplicity of $\lambda$ as an eigenvalue of $\lminusva{}{a}$ is equal to the multiplicity of $2(\lambda-k-1)$ as an eigenvalue of $A A^{\T}$.

Notice that $AA^{\T} = J + (d-2)I_d$, where $J\in \Rea^{d\times d}$ is the all-ones matrix.
    It is easy to check that the all-ones vector, $\mathbf{1}_d\in \Rea^{d}$, is an eigenvector of $A A^{\T}$ with eigenvalue $2d-2$, and every vector in $\Rea^{d}$ orthogonal to $\mathbf{1}_d$ is an eigenvector of $A A^{\T}$ with eigenvalue $d-2$. 

Finally, note that $2d-2= 2(\lambda-k-1)$ for $\lambda=d+k=n$. Hence, $\lambda=n$ is a simple eigenvalue of $\lminusva{}{a}$. Similarly, we have $d-2=2(\lambda-k-1)$ for $\lambda= d/2+k= n-d/2$. Thus, $\lambda=n-d/2$ is an eigenvalue of $\lminusva{}{a}$ with multiplicity $d-1$, as wanted.

\end{proof}

\section{Concluding remarks}\label{sec:concluding_remarks}

In this paper, we investigated the effect of the blow-up operation on the stiffness matrix spectra of graphs and its application to the study of the $d$-dimensional algebraic connectivity of complete bipartite graphs.  
In particular, in Theorem \ref{thm:knm}, we showed that there is a constant $c_d>0$ such that, for large enough $n$ and $m$,  $a_d(K_{n,m})\ge c_d\cdot  \min\{n,m\}$. 
It would be interesting to try to extend our methods to the case of complete multipartite graphs.

\begin{conjecture}
     Let $d\geq 1$ and $k\geq 2$. Then, there exist $c_{d,k}, M_{d,k} > 0$ such that for all $n_1, n_2, \ldots, n_k \geq M_{d,k}$,
     \[
     a_d(K_{n_1, n_2, \ldots, n_k}) \geq c_{d,k}\cdot \min\{n-n_1, n-n_2, \ldots, n-n_k\},
     \]
     where $K_{n_1, n_2, \ldots, n_k}$ is the complete $k$-partite graph with sides of size $n_1,n_2,\ldots,n_k$,  respectively, and $n=n_1+\cdots+n_k$.
\end{conjecture}
Note that $a_1(K_{n_1,n_2,\ldots,n_k})=\min\{n-n_1, n-n_2, \ldots, n-n_k\}$ (see, for example, \cite[Proposition 3.2]{klee2022eigenvalues}). Let us mention that, to the best of our knowledge, the problem of determining, for given integers $n_1,n_2,\ldots,n_k$ and $d\ge 2$, whether the complete $k$-partite graph $K_{n_1,n_2,\ldots,n_k}$ is $d$-rigid remains open for $k\ge 3$.
   
The problem of determining the exact value of $a_d(K_{n,m})$, for $d\ge 2$, remains unresolved. In the special case $d=2$ and $n=m=3$, we conjecture, based on computer experiments, that the lower bound proved in Theorem \ref{thm:k33_tight} is tight.

\begin{conjecture} Let $\lambda\approx 0.6903845$ be the unique positive real root of the polynomial $176 x^4-200 x^3+47 x^2+18 x-9$. Then,
    \[a_2(K_{3,3}) = 2(1-\lambda)\approx 0.6192309.\] 
\end{conjecture}
The best currently known bound upper bound on $a_2(K_{3,3})$, which follows from Proposition \ref{prop:minimally_rigid} (noting that $K_{3,3}$ is minimally $2$-rigid), is $a_2(K_{3,3})\le 1$.

\bibliography{main}

\end{document}